\documentclass[11pt,reqno]{amsart}

\usepackage{amsfonts} 
\usepackage{amsmath,amscd}
\usepackage{amssymb}
\usepackage{tikz-cd}
\usepackage{amsthm}
\usepackage{cases}

\usepackage{galois}
\usepackage{chemarrow}
\usepackage[all]{xy}
\usepackage{graphicx}
\usepackage[colorlinks,
            linkcolor=magenta,
            anchorcolor=black,
            citecolor=red
            ]{hyperref}	
\usepackage{mathrsfs}
\usepackage{extarrows}
\usepackage{texnames}
\usepackage{textcase}

\makeatletter
\def\section{%
  \@startsection{section}{1}
    {\z@}
    {2.0ex plus 0.8ex minus .1ex}
    {1.0ex plus .2ex}
    {\large\bfseries\boldmath\centering\MakeTextUppercase}%
}
\makeatother

\def\grad{\nabla}

\newcommand{\si}{\sigma}

\newcommand{\ep}{\epsilon}

\newcommand{\al}{\alpha}

\newcommand{\vol}{V}

\newcommand{\dist}{{\rm dist}}
\newcommand{\lap}{\Delta}

\newcommand{\be}{\beta}
\newcommand{\boundary}{\partial}
\newcommand{\ti}{\tilde}

\newcommand{\partt}{\frac{\partial}{\partial t}}
\newcommand{\upto}{\nearrow}

\def\Sc{{\rm  R}}
\def\Rm{{\rm  Rm}}
\def\si{\sigma}

\def\Ric{{\rm Ric}}
\def\Rc{{\rm Ric}}
\def\Ricci{{\rm Ric}}

\def\boundary{ \partial }
\newcommand{\N}{ \mathbb N}
\newcommand{\R}{ \mathbb R}

\newtheorem{thm}{Theorem}[section]
\newtheorem{lem}[thm]{Lemma}

\newtheorem{rem}[thm]{Remark}

\newcommand*{\QEDB}{\hfill\ensuremath{\square}}
\numberwithin{equation}{section}

\begin{document}
\title[\sc\small Estimates for  Ricci flows  with $L^p$ bounded scalar curvature ]{Integral curvature estimates for solutions to Ricci flow with $L^p$ bounded scalar curvature}
\author[\sc\small J\MakeLowercase{iawei} L\MakeLowercase{iu and} M\MakeLowercase{iles} S\MakeLowercase{imon}]{\sc\large J\MakeLowercase{iawei} L\MakeLowercase{iu and} M\MakeLowercase{iles} S\MakeLowercase{imon}}
\address{Jiawei Liu\\  School of Mathematics and Statistics \\  Nanjing University of Science \& Technology \\  Xiaolingwei Street 200\\ Nanjing 210094 \\ China\\} \email{jiawei.liu@njust.edu.cn}
\address{Miles Simon\\ Institut f\"ur Analysis und Numerik\\ Otto-von-Guericke-Universit\"at Magdeburg\\ universit\"atsplatz 2\\ Magdeburg 39106\\ Germany\\} \email{miles.simon@ovgu.de}
\thanks{Both  authors were  supported by grants in the Special Priority Program SPP 2026 ``Geometry at Infinity" of  the German Research Foundation (DFG). }

\maketitle
\vskip -3.99ex

\centerline{\noindent\mbox{\rule{3.99cm}{0.5pt}}}

\vskip 5.01ex

\ \ \ \ {\bf Abstract: }
In this   paper we  prove  {\bf localised weighted}  curvature   integral estimates for solutions to the Ricci flow 
in the setting of a  smooth four dimensional Ricci flow or a closed $n$-dimensional Kähler Ricci flow. 
These integral   estimates improve and extend  the integral curvature estimates shown by the second author  in an earlier paper. If  
the scalar curvature is uniformly bounded in the spatial $L^p$ sense for some $p>2,$ then the estimates imply a uniform bound on the spatial $L^2$ norm of the Riemannian curvature  tensor. Stronger integral estimates are shown to hold if one further assumes a weak non-inflating condition, or we restrict to  closed manifolds.

\medskip


\section{Introduction}\label{sec:intro}

In this   paper we  prove  {\bf localised weighted}  curvature   integral estimates for solutions to the Ricci flow,  
which generalise and improve  those proved in \cite{MSIM2}. The Ricci flow, $\partt g(t) = -2\Rc_{g(t)}$  was first  introduced and studied by 
R. Hamilton in \cite{Ham}.\\

Here, and in the sequel  paper \cite {LiuSim2},  we are interested  in  the setting: \\

 \noindent ({\bf A}): 
$(M^n,g(t))_{t\in [0,T)}, 0<T < \infty$  is a  smooth 
$n$-dimensional  solution   to Ricci flow   with  $\inf\limits_{M \times [0,T)}
\Sc_{g(t)}  > -\infty$ for all $t\in  [0,T)$ and $N$  is a smooth, connected    $n$-dimensional submanifold   with smooth boundary $\boundary N$, $N$  compactly contained in $M$,   such that 
$  \Omega_0 :=   B_{g(0)}(\boundary N,\si)  $ has compact closure  in $M$     and
$ \sup\limits_{    x \in   { \Omega_0 }, t  \in [0,T) } |\Rm_{g(t)}|_{g(t)} < \infty .$    
\\

{ Using Perleman's Pseudolocality Theorem (see \cite{Top} Lemma A.4),  compactness of $\overline{\Omega_0}$ and the fact that the solution is smooth, we can scale the solution once by a large constant $C$, that is we consider $(M, Cg(\frac{t}{C}))_{t\in [0,T C)}$  in place of $(M,g(t))_{t\in [0,T)}$ to  arrive at the basic setting, which we often assume in this paper: }
\\

\vspace{0.5cm}
\noindent ({\bf B}):   $(M^n,g(t))_{t\in [0,T)}, 2\leq T < \infty$  is a smooth 
$n$-dimensional  solution   to Ricci flow   with  $\Sc_{g(t)} \geq -1$ for all $t\in  [0,T)$ and $N$  is a connected, smooth   $n$-dimensional submanifold   with smooth boundary $\boundary N$, $N$  compactly contained in $M$,   such that 
$  \Omega :=  \cup_{s=0}^{T} B_{g(s)}(\boundary N,10)  $ has compact closure  in $M$  and  $\sup\limits_{    x \in  \overline{ \Omega }, t  \in [0,T) } |\grad^k \Rm_{g(t)}|_{g(t)} \leq 1$   for all $k\in\{1,\ldots,4\},$\\     $ \sup\limits_{    x \in  \overline{ \Omega }, t  \in [0,T) } |\grad^k \Rm_{g(t)}|_{g(t)}\leq C_k< \infty ,$ for all $k\geq 5,$  $k\in \N,$
and \\$\sup\limits_{x \in \Omega \cup N, t \in [0,1] }  |\Rm_{g(t)}|_{g(t)}\leq 1,$ 
\vspace{1cm}
 
 \vspace{0.5cm}
 { 
In the case that $(M^n,g(t))_{t\in [0,T)}$ is a   closed, smooth Kähler solution to the  Ricci flow,
we may  change  the condition ${\bf (A)}$ to\\

\noindent ({\bf  C}): 
$(M^n,g(t))_{t\in [0,T)}, 0<T < \infty$  is a  smooth, closed  
$n$-dimensional  solution   to K\"ahler-Ricci flow   with  $\inf\limits_{M \times [0,T)}
\Sc_{g(t)} > -\infty$ for all $t\in  [0,T)$ and $N$  is a smooth, connected $n$-dimensional submanifold   with smooth boundary $\boundary N$, $N$  compactly contained in $M$,   such that 
$  \Omega_0 :=   B_{g(0)}(\boundary N,\si)  $ has compact closure  in $M$     and there exists a constant $c>0$ such that $\frac{1}{c} g_0 \leq g(t) \leq c g_0 $ for all $t\in [0,T).$ \\ 
 It is still then possible, using compactness of
 $\overline{\Omega_0},$ smoothness of the solution and  Corollary 1.2 of \cite{SHWE} (see also \cite{JWLXZ})
 to scale the solution once, by a large constant, so that the new solution satisfies ({\bf B}). }  
 
We are interested in showing integral formulae   on the  local region $N$   for  $t\in [0,T)$  
 where a neighbourhood of   $\boundary N$ is regular :  $(N,g(t))_{t\in [0,T)}$ is also a solution to Ricci flow, which possibly becomes singular at time $T$, but it is 
regular at and near its boundary.

In the case that  we restrict to real four dimensional Ricci flow solutions or K\"ahler-Ricci flow solutions of any dimension,   we obtain the following:

\begin{thm}\label{firstmainwieIntro}
For $n\in\N,$ $1\leq {\rm V} < T < \infty,$ 
let   $(M^{n},g(t))_{t\in [0,T)}$  be a smooth,  real  solution
to Ricci flow, $\partt g(t) = -2\Rc_{g(t)}.$  
We assume that $n=4$ or that $(M^n,g_0)$ is Kähler ( complex  dimension $\frac n 2,$ real dimension $n$ ) and closed.
Assume further that  $\alpha\in(0,\frac{1}{12})$   and  $N  \subseteq M$   
  is  a  smooth  connected  real $n$-dimensional   submanifold with boundary, and  $N$ and $(M^{n},g(t))_{t\in [0,T)}$   are  as in ({\bf B}). 
Then there exists a constant $\hat c_0 = \hat c_0(N,g_0,T,\Omega,   g|_{\Omega} ) < \infty$  such that for 
$b \geq \hat b(\al,T) =  \frac{2000 T}{\al}$ we have    for all $r<s$ with $ r,s \in [0,{\rm V}]$ 
 that 
\begin{eqnarray}\label{RIntEstIntro} 
 &&  \int_N \frac{|\Rc_{g(s)}|_{g(s)}^2}{2{\rm V}  + \Sc_{g(s)}({\rm V}-s)^{1-\al}} d\vol_{g(s)} \cr
 &&  \ \ \  +    \int_r^s  \int_N   ({\rm V}-t)^{1-\alpha}\frac{|\Rc_{g(t)}|_{g(t)}^4}{ ( 2{\rm V}  + \Sc_{g(t)}({\rm V}-t)^{1-\al}   )^2}   d\vol_{g(t)}    dt \cr
   && \leq  e^{b({\rm V}-r)^{\al}} \frac{1}{\al}\hat c_0 (s-r)^{\al}  + 
  e^{b({\rm V}-r)^{\al}} \int_N \frac{|\Rc_{g(r)}|_{g(r)}^2 }{ 2{\rm V}  + \Sc_{g(r)} ({\rm V}-r)^{1-\al} } d\vol_{g(r)}   \cr
    && \ \ \   +  e^{b({\rm V}-r)^{\al}}  \int_r^s  \int_N  \frac{1 }{({\rm V}-t)^{1-\alpha}}  \Sc_{g(t)}^{2+12\alpha}d\vol_{g(t)}    dt.
\end{eqnarray}

\end{thm}

The proof of this basic estimate requires that we estimate the spatial $L^2$ norm of the full Riemannian curvature tensor by terms  only involving  the spatial $L^2$ norm of the Ricci and scalar curvature, quantities involving only the time zero metric on $N,$ and bounded quantities on the boundary of $N$. In four dimensions this is achieved with the help of the generalised Gauss-Bonnet Theorem, as it was in the paper \cite{MSIM1}, although here we require the version of the generalised Gauss-Bonnet Theorem with boundary. 
For the K\"ahler case, we use   estimates for the $L^2$ norm of the full curvature which are shown in this paper and are valid in any dimension. These estimates are proven with the help of formulae from Apte  \cite{APTE}  :

 \begin{thm}\label{KaehlerIntEstCorIntro}
Let $(M,g(t))_{t\in[0,T)}$ with $T<\infty$ be a smooth closed solution to the Ricci flow  and assume 
$(M,g_0)$ is a closed K\"ahler manifold  with complex dimension $m$ and   $N  \subseteq M$ is   a  smooth   connected   $m$-dimensional complex   submanifold with boundary $\boundary N,$ where $N  \subseteq M,$ $\Omega$   and $(M^{n},g(t))_{t\in [0,T)}$    are  as in ({\bf B}), with $n=2m$.
 Then  there exists a $C=C(m,N,g(0), \Omega, g|_{\Omega},   T   ) < \infty$  such that
\begin{eqnarray}
 \int_N |\Rm_{g(t)}|_{g(t)}^2 d\vol_{g(t)} \leq  \int_N \Sc_{g(t)}^2 d\vol_{g(t)} +C
 \end{eqnarray}
 for all $t\in [0,T)$. 
\end{thm}

If we further assume that the scalar curvature is bounded in the spatial  $L^{2+v}$ sense, for some $v\in (0,1)$  
then we obtain further integral estimates:

\begin{thm}\label{secondmainweiIntro}
For $n\in\N,$ $1\leq {\rm V} < T < \infty,$  $\alpha\in(0,\frac{1}{12})$  
let   $(M^{n},g(t))_{t\in [0,T)}$  be a smooth,  real  solution
to Ricci flow and  assume 
that $n=4$ or that $(M^n,g_0)$ is Kähler ( complex  dimension $\frac n 2,$ real dimension $n$ ) and closed.
Assume further that    $N  \subseteq M$   
  is  a  smooth  connected  real $n$-dimensional   submanifold with boundary, and  $N$ and $(M^{n},g(t))_{t\in [0,T)}$   are  as in ({\bf B}) and that 
$$\int_N |\Sc|_{g(t)}^{2+12\alpha} d\vol_{g(t)}  \leq C_0< \infty$$ for all $t\in [0,T)$ for some constant $C_0\in \R^{+}.$ 
Then, for all $l  \in [0,T)$ we have  
\begin{eqnarray} 
&& i)   \int_{N} |\Rm_{g(l)}|_{g(l)}^2 d\vol_{g(l)}   
  \leq \hat c_1\label{RiemestgoodIntro}, \\\nonumber
   && ii)   \int_{ {\rm V}-2s}^{{\rm V}-s} \int_{N \cap B_{g(t)}(p,r) } |\Rc_{g(t)}|_{g(t)}^4 d\vol_{g(t)} dt\\
   &&\ \ \   \leq \hat c_1 \frac{1}{s^{1-\alpha} } + 
  \hat c_1 \sup \{     |\Rc_{g(t)}|_{g(t)}^2 \ | \   t \in [{\rm V}-2s,{\rm V}-s] , x    \in B_{g(t)}(p,r) \} s^{1+\alpha}
 \end{eqnarray}
  for all $p \in N,$   if  $s \in (0,1)$ and ${\rm V}-3s>0$, where  $\hat c_1$ is a constant depending on $ n,N,g(0), \Omega, g|_{\Omega},  \frac{1}{\al}, T, C_0.$ 
\end{thm}

\begin{rem}
  
If we further assume that    $|\Rc|_{g(t)} \leq \frac{C}{t-({\rm V}-3s) } $ on $B_{g(t)}(p,r)$ for $t \in ({\rm V}-3s,{\rm V}-s]$, where  ${\rm V}-3s>0$,   as is the case for example if the  region   $B_{g({\rm V}-3s)}(p,2r)$ is almost euclidean and $M$ is closed, and $|s|\leq  c(n)r^2$ (see Perleman's Pseudolocality Theorem, Theorem 10.1  of \cite{GP}),
then   estimate  ii)  implies
$$ \int_{ {\rm V}-2s}^{{\rm V}-s} \int_{N \cap B_{g(t)}(p,r) } |\Rc_{g(t)}|_{g(t)}^4 d\vol_{g(t)} dt \\
 \leq  \frac{\hat c_1+C}{s^{1-\alpha}}\ \text{for}\ s \in (0,1). $$ 
 \end{rem}

In the case that $M$ is closed,
it is known due to a result of Bamler (\cite{Bam}, Theorem 8.1),
that a non-inflating estimate holds ${\rm Vol}_{g(t)}(B_{g(t)}(x,r)) \leq \si_1r^n$ will be satisfied for all $r \in [0,1]$ and some $\si_1>0,$ for all $ t\in [0,T)$ for all $x \in  N$.
Under a   weaker  non-inflating assumption in the general case, namely that  there exists some $p\in N,$ $V \in [1,T]$ such that 
\begin{eqnarray}
{\rm Vol}_{g(t)}(B_{g(t)}(p, \sqrt{V-t} ) ) \leq \si_1 |V-t|^2  \label{NonExpandTIntro}\ \text{ for all  } \ t\in [V-1,V),  
\end{eqnarray} 
 along with the conditions 
 from Theorem \ref{secondmainweiIntro}, we prove the following integral estimate. 

\begin{thm}\label{AnotherIntIntro}
For $n\in\N,$ $1\leq {\rm V} \leq T < \infty,$   $\alpha\in(0,\frac{1}{12})$  
let   $(M^{n},g(t))_{t\in [0,T)}$  be a smooth,  real  solution
to Ricci flow and  assume 
that $n=4$ or that $(M^n,g_0)$ is Kähler ( complex  dimension $\frac n 2,$ real dimension $n$ ) and closed.
Assume further that    $N  \subseteq M$   
  is  a  smooth  connected  real $n$-dimensional   submanifold with boundary, and  $N$ and $(M^{n},g(t))_{t\in [0,T)}$   are  as in ({\bf B}) and that 
$$\int_N |\Sc|_{g(t)}^{2+12\alpha} d\vol_{g(t)}  \leq C_0< \infty$$ for all $t\in [0,T)$ for some constant $C_0\in \R^{+},$ 
and that there exist  a $\si_1>0,$ $p\in N,$  such  that  the weak non-inflating condition \eqref{NonExpandTIntro} is satisfied.   
Then we have:
\begin{eqnarray}
&& \int_{S}^V \int_{B_{g(t)}(p,\sqrt{V-t})} |\Rc_{g(t)}|_{g(t)}^{2+ \sigma }  d\vol_{g(t)}  dt 
 \leq   \hat c_2    (V-S)^{1 + \frac{\al}{16}} \label{RiccestgoodIntro}
 \end{eqnarray} 
   for all  $\sigma>0$ sufficiently small   ( $\si\leq \al^3$ suffices   ) and   $S \in [V-1,V),$  
  where $\hat c_2< \infty$ depends on  $  \si_1,  n,N,g(0), \Omega, g|_{\Omega},  \frac{1}{\al}, T, C_0.$   
   For $V= T,$ we mean  
 \begin{eqnarray}
&& \lim_{V \upto T} \int_{S}^V \int_{B_{g(t)}(p,\sqrt{T-t})} |\Rc_{g(t)}|_{g(t)}^{2+\sigma}  d\vol_{g(t)}  dt  
  \leq  \hat c_2   (T-S)^{1 + \frac{\al}{16}} \label{Riccestgood2Intro}
 \end{eqnarray}
  
\end{thm} 
\begin{rem} \label{noninflrem}
{ We remark here, that   this is not the usual non-inflating condition when $n> 4$,  and is a weaker condition in that case: The usual non-inflating condition would be
${\rm Vol}_{g(t)}(B_{g(t)}(p,r)) \leq \si r^n$ for all $r\leq 1$ and this implies 
${\rm Vol}_{g(t)}(B_{g(t)}(p,\sqrt{V-t})   ) \leq \si_1 |V-t|^{n/2} \leq \si_1 |V-t|^2$ when
$|V-t|\leq 1$.}\\
Hence, in the case that $M$ is closed,  this non-inflating condition will  hold (and need not be assumed),  due to  the result of Bamler  (\cite{Bam}, Theorem 8.1). 
\end{rem}

In the sequel to this paper, \cite{LiuSim2}, we use these estimates and other results and methods to  prove 
(see  also 
 \cite{SHA}) \begin{itemize}
\item[(a)]  non-collapsing  for solutions  $(M^n,g(t))_{t\in [0,T)},$ $n\in \N,$   $T< \infty$  
 when $ M$ is closed,  $N \subseteq M$,  $N$ and $(M^{n},g(t))_{t\in [0,T)} $    are  as in ({\bf B}),  and $$\sup\limits_{t\in [0,T)} \int_{N} |\Sc_{g(t)}|^{\frac{n}{2}+v}d\vol_{g(t)} < \infty  $$ for some $v\in(0,1)$
 (non-inflating  estimates are already known to hold when $M$ is closed due to Bamler (Theorem 8.1, \cite{Bam})). 
 \item[(b)]  distance estimates and $C^0$ orbifold convergence of $(N^4,g(t))$ 
 as $t \upto T< \infty$   for four   dimensional real closed solutions $(M^4,g(t))_{t\in [0,T)}$    when   $N$ and $(M^{4},g(t))_{t\in [0,T)}$    are  as in ({\bf B}) and  satisfy 
 $$\sup_{t\in [0,T)} \int_{N} |\Sc_{g(t)}|^{2+v}dV_{g(t)} < \infty  $$ for some $v\in (0,1)$,
 thus generalising the results of Bamler-Zhang (\cite{RBQZ, RBQZ2}) and Simon  (\cite{MSIM1, MSIM2}), where the case that the  scalar curvature remains uniformly bounded on   $[0,T)$ was considered  
\item[(c)] if $(M^4,g(t))_{t\in [0,T)}$  is a closed, real solution to Ricci flow 
 and $\sup\limits_{t\in [0,T)} \int_{M} |\Sc_{g(t)}|^{2+v}dV_{g(t)} < \infty  $ for some $v\in (0,1)$,
then the solution   can be extended for a short time $\si >0$ to
 $[0,T+ \si)$ using the orbifold Ricci flow, thus extending the results of \cite{MSIM2}, where an analogous result was shown
 in the case that the scalar curvature remains uniformly bounded  on $[0,T).$  
\end{itemize}

\section{Basic weighted  integral estimates  for the  Ricci flow in four dimensions}\label{sec:3}

In \cite{MSIM1},   the second author   proved an  integral estimate for solutions to the  Ricci flow in the four dimensional setting, when  the initial value $g(0)$ satisfies     $\inf_{M}\Sc_{g(0)}>-1.$ There it was shown that  
\begin{equation*}\label{basicest}
\begin{split}
&\ \ \ \int_{M} \frac{|\Ric_{g(S)}|^2}{\Sc_{g(S)}+2} dV_{g(S)}+\int_0^S\int_M \frac{|\Ric_{g(t)}|^4}{(\Sc_{g(t)}+2)^2} dV_{g(t)}dt\\
&\leqslant 2^{10}e^{64S}\Big(\chi+ \int_{M} \frac{|\Ric_{g(0)}|^2}{\Sc_{g(0)}+2} dV_{g(0)} +  \int_0^S\int_M\Sc_{g(t)}^2dV_{g(t)}dt \Big).
\end{split}
\end{equation*}
This was achieved by examining the evolution equation of   $\int_{M}  \ell(s) dV_{g(s)}$ for the integrand
$  \ell (s) := \frac{|\Ric_{g(s)}|_{g(s)}^2}{\Sc_{g(s)}+2},$ in conjunction with
the  generalised Gau{\ss}-Bonnet formula
\begin{equation*} 
\int_M (|\Rm_{g}|_g^2-4|\Ric_{g}|_g^2+\Sc_g^2)\ dV_g=2^5\pi^2\chi
\end{equation*}
 for closed four dimensional Riemannian manifolds, where $\chi=\chi(M)$ is the  Euler characteristic.   

In this section we prove localised weighted versions  of these integral estimates valid for  regions  of manifolds which are evolving   by  the Kähler Ricci flow in any dimension or   Ricci flow in four dimensions. The localisation is achieved by assuming that the  
regions we  consider  are geometrically controlled near their boundaries. 
In the Kähler setting, one also requires that the solutions being considered come from a potential, as is the case when  the manifold is closed. The weights we consider   are time like  ones: 
We consider solutions to the Ricci flow $(M,g(t))_{t\in [0,T)}$ with $1\leq {\rm V}< T  < \infty$ and  $\Sc_{g(t)} +  1 > 0$ for all $t\in [0,T)$ and  the integrand 
\begin{eqnarray*}
f(t):= \frac{ | \Rc_{g(t)}|_{g(t)}^2}{ 2{\rm V}   + \Sc_{g(t)}({\rm V}-t)^{1-\al}}.
\end{eqnarray*}

As a first step towards showing these weighted estimates, we present  a local result which is valid  for solutions and regions of the type described  above, in any dimension, without making the assumption that the solution is a K\"ahler-Ricci flow solution or four dimensional.

\begin{thm}
For $n\in\N,$ $1\leq {\rm V} < T < \infty,$ 
let   $(M^{n},g(t))_{t\in [0,T)},$  be a smooth,  real  solution
to Ricci flow, $\partt g(t) = -2\Rc_{g(t)}.$  
Assume further that   
 $\al \in (0,\frac 1 4),$   $N  \subseteq M,$ $\Omega$   and $(M^{n},g(t))_{t\in [0,T)}$    are  as in ({\bf B}). 
Then there exists a constant $\hat c = \hat c(N,\Omega, g|_{\Omega})< \infty$  such that for  $ b \geq   \frac{10 T} {\al}$ and 
 $L_{\rm V}(t):= 2{\rm V}   + \Sc_{g(t)}({\rm V}-t)^{1-\al}$  we have 
 
\begin{eqnarray}\label{FirstRIntEst} 
 && e^{b({\rm V}-s)^{\al}} \int_N \frac{|\Rc_{g(s)}|_{g(s)}^2}{L_{\rm V}(s)} d\vol_{g(s)} + \int_r^s  \int_N  
 \frac{\al b e^{b({\rm V}-t)^{\al}}}{2({\rm V}-t)^{1-\alpha} }  \frac{|\Rc_{g(t)}|_{g(t)}^2}{L_{\rm V}(t)} d\vol_{g(t)}dt \cr
&& \ \ \ +     \frac{9}{8}\int_r^s  \int_N e^{b({\rm V}-t)^{\al}}  ({\rm V}-t)^{1-\alpha}\frac{|\Rc_{g(t)}|_{g(t)}^4}{L_{\rm V}^2(t)}   d\vol_{g(t)}    dt  \cr
    && \leq e^{b({\rm V}-r)^{\al}} \int_N \frac{|\Rc_{g(r)}|_{g(r)}^2}{L_{\rm V}(r)} d\vol_{g(r)}   + e^{b({\rm V}-r)^{\al}} \hat c (s-r)   \cr
    &&\ \ \     + 
  \int_r^s  \int_N  \frac{e^{b({\rm V}-t)^{\al}}  }{({\rm V}-t)^{1-\alpha}} \Big(  8|\Rm_{g(s)}|_{g(s)}^2 +  T\Sc_{g(t)}^{2+4\alpha}(t) \Big) d\vol_{g(t)} dt 
\end{eqnarray} 
  for all $r<s$ with $ r,s \in (0,{\rm V}]$.
\end{thm}
 \begin{rem} The integrands  in the above are well defined  for all $t\in [0,{\rm V})$ since $ L_{\rm V} (t)=   
   2{\rm V}   +  \Sc_{g(t)}({\rm V}-t)^{1-\al} \geq    {\rm V} + ({\rm V}-t)^{1-\al}  +  \Sc_{g(t)}({\rm V}-t)^{1-\al} =
{\rm V}  + (1 +  \Sc_{g(t)})({\rm V}-t)^{1-\al} \geq  {\rm V} \geq 1$ and $\al \in (0,1).$
\end{rem} 
\begin{rem}
In the case that $M$ is closed and $\Sc_{g(0)}\geq -1,$ the condition $\Sc_{g(t)} \geq -1$ for all $t\in [0,{\rm T})$ is always satisfied, in view of the strong maximal principle, and so may be removed from the list of assumptions.
\end{rem}   
\begin{rem}
The  values  $e^{b ({\rm V}-t)^{\al}}$ with $t\in[r,s]$ appearing in the formula above can be bounded from above and below by $1 \leq e^{b ({\rm V}-t)^{\al}} \leq 2$  if we restrict  to 
$s,r \in (S,{\rm V})$ where $   |S-{\rm V}| \leq  (\frac{\log 2}{b} )^{\frac 1 {\al}}.$
Alternatively they may be bounded from above and below by 
$1 \leq e^{b ({\rm V}-t)^{\al}} \leq e^{b {\rm V}^{\al}}  $    for all $ t\in [0,{\rm V}).$  
\end{rem}

\begin{proof}
We use the notation  
$L_{\rm V}(t):=   2{\rm V}  + \Sc_{g(t)}({\rm V}-t)^{1-\al}$ and $ f(t)  := \frac{ | \Rc_{g(t)}|^2_{g(t)}}{ L_{\rm V}(t)}$ in the following. As mentioned in the Remark above, we always have $L_{\rm V}(t) \geq 1$. 
The evolution equation for $  \int_N f(t)  d\vol_{g(t)}$ is
\begin{eqnarray}
&&  \partt (\int_N f(t) d\vol_{g(t)} )\cr
&& =  \int_{N} \partt(\frac{ | \Rc_{g(t)}|_{g(t)}^2}{ 2{\rm V}  + \Sc_{g(t)}({\rm V}-t)^{1-\al}})d\vol_{g(t)} + \int_N f(t) \partt d\vol_{g(t)}\cr
&& = \int_N ( \lap_{g(t)} f(t) -\frac{2|P(t)|^2}{L^3_{\rm V}(t)}  + \frac{4\Rm_{g(t)}(\Rc_{g(t)},\Rc_{g(t)}) }{L_{\rm V}(t)}
 -  \frac{2|\Rc_{g(t)}|_{g(t)}^4({\rm V}-t)^{1-\al}}{L_{\rm V}^2(t)}) d\vol_{g(t)}\cr
&& \  \ \  \ \ \
+ (1-\al) \int_N \frac{|\Rc_{g(t)}|_{g(t)}^2 \Sc_{g(t)} ({\rm V}-t)^{-\al}}{L^2_{\rm V}(t)}d\vol_{g(t)} - \int_N \Sc_{g(t)} \frac{|\Rc_{g(t)}|_{g(t)}^2}{L_{\rm V}(t)}d\vol_{g(t)}\cr
&& \leq   -\int_{\boundary N}  <\grad_{g(t)} f(t), \nu(t)>_{g(t)} d\vol_{g(t)} 
 + \int_N \frac{8|\Rm_{g(t)}|_{g(t)}^2}{({\rm V}-t)^{1-\al}} d\vol_{g(t)}  \cr
 &&   \ \ \  -\int_N \frac{3|\Rc_{g(t)}|_{g(t)}^4({\rm V}-t)^{1-\al}}{2L^2_{\rm V}(t)} d\vol_{g(t)}  
+ (1-\al) \int_N \frac{|\Rc_{g(t)}|_{g(t)}^2 \Sc_{g(t)}({\rm V}-t)^{-\al}}{L_{\rm V}^2(t)}d\vol_{g(t)}\cr
&& \ \ \ - \int_N \Sc_{g(t)} \frac{|\Rc_{g(t)}|_{g(t)}^2}{L_{\rm V}(t)}d\vol_{g(t)}\cr
&&= I_0(g|_{B_1(\boundary N)} ) +  I_1 -I_2+I_3 +I_4
\end{eqnarray}
where 
\begin{eqnarray}
&\ I_0(g|_{B_1(\boundary N)} )=-\int_{\boundary N} <\grad_{g(t)} f, \nu(t)>_{g(t)}d\vol_{g(t)},\\
&\ I_1= \int_N \frac{8|\Rm_{g(t)}|_{g(t)}^2}{({\rm V}-t)^{1-\al}} d\vol_{g(t)},\\
&\ I_2= \int_N \frac{3|\Rc_{g(t)}|_{g(t)}^4({\rm V}-t)^{1-\al}}{2L_{\rm V}(t)^2} d\vol_{g(t)}, \\
&\ I_3= (1-\al) \int_N \frac{|\Rc_{g(t)}|_{g(t)}^2 \Sc_{g(t)}({\rm V}-t)^{-\al}}{L_{\rm V}(t)^2}d\vol_{g(t)},\\
&\ I_4= - \int_N \Sc_{g(t)} \frac{|\Rc_{g(t)}|_{g(t)}^2}{L_{\rm V}(t)}d\vol_{g(t)},\\
and &\ P(t)= (\grad \Ricci_{g(t)})(\Sc_{g(t)} ({\rm V}-t)^{1-\al} +2{\rm V} ) - (\grad ({\rm V}-t)^{1-\al}\Sc_{g(t)})( \Ricci_{g(t)}).
\end{eqnarray}
By assumption $\bf (B)$, we have  $I_0(g|_{B_1(\boundary N)} )  \leq|\int_{\boundary N} <\grad_{g(t)} f(t), \nu(t)>_{g(t)} d\vol_{g(t)} | \leq \hat c< \infty$  for all $t\in [0,T),$ where $\hat c= \hat c(N,\Omega,g|_{\Omega})< \infty.$   

In the following we will estimate the integrals  $I_3,I_4.$   

We consider  the term $I_4$ first :

\begin{eqnarray}
I_4 &=& - \int_N \Sc_{g(t)}  \frac{|\Rc_{g(t)}|_{g(t)}^2}{L_{\rm V}(t)}d\vol_{g(t)}\cr
&\leq& \int_N  \frac{|\Rc_{g(t)}|_{g(t)}^2}{L_{\rm V}(t)}d\vol_{g(t)}
- \int_{N \cap \{ \Sc_{g(t)} \geq 0\}}  |\Sc_{g(t)}|   \frac{|\Rc_{g(t)}|_{g(t)}^2}{L_{\rm V}(t)}d\vol_{g(t)},
\end{eqnarray}
the last integral being no larger than zero. 

$I_3$ will be estimated in two steps.
First we consider the time dependent sets 
$ \{ 
x  \in N  \ | \  \Sc_{g(t)}(x)  ({\rm V}-t)^{1-2\al} \leq 1\} =: 
N \cap \{ \Sc_{g(t)}  ({\rm V}-t)^{1-2\al}  \leq 1\}.$
Then
\begin{eqnarray}
&&  (1-\al) \int_{ N \cap \{ \Sc_{g(t)}   ({\rm V}-t)^{1-2\al}  \leq 1\} }  \frac{|\Rc_{g(t)}|_{g(t)}^2 \Sc_{g(t)}({\rm V}-t)^{-\al}}{L^2_{\rm V}(t)}d\vol_{g(t)} \cr
&& \leq   (1-\al) \int_{ N \cap \{ \Sc_{g(t)}  ({\rm V}-t)^{1-2\al}  \leq 1\}}  \frac{|\Rc_{g(t)}|_{g(t)}^2}{({\rm V}-t)^{1-\al} L^2_{\rm V}(t)}d\vol_{g(t)} \cr
&& \leq   \int_N \frac{|\Rc_{g(t)}|_{g(t)}^2}{({\rm V}-t)^{1-\al} L_{\rm V}(t)} d\vol_{g(t)},
\end{eqnarray}
where we have used $L_{\rm V}(t) \geq 1.$ 
For the time dependent sets $    Z_t:= N \cap \{ \Sc_{g(t)}  ({\rm V}-t)^{1-2\al}  \geq 1\}$   
we define $m(t):= {\rm Vol}_{g(t)}(  Z_t ).$
Then  we see  that 
\begin{equation}
m(t) \frac{1}{ ({\rm V}-t)^{2-4\al}} \leq \int_{Z_t} \Sc_{g(t)}^2 d\vol_{g(t)}\leq  (\int_{Z_t} |\Sc_{g(t)}|^{2+\be}d\vol_{g(t)})^{\frac{2}{2+\be}} (m(t))^{\frac{\be}{2+\be}},
\end{equation} 
 that is,
\begin{equation}
m(t)^{\frac{2}{2+\be}} \leq ({\rm V}-t)^{2-4\al} (\int_{Z_t}|\Sc_{g(t)}|^{2+\be}d\vol_{g(t)})^{\frac{2}{2+\be}},
\end{equation} 
 and hence
\begin{equation}m(t)  \leq ({\rm V}-t)^{ (1-2\al)(2+\be)} \int_{Z_t} |\Sc_{g(t)}|^{2+\be}d\vol_{g(t)}.
\end{equation} 
This leads to 
\begin{eqnarray}
&&(1-\al) \int_{_{Z_t}}  \frac{|\Rc_{g(t)}|_{g(t)}^2 \Sc_{g(t)}({\rm V}-t)^{-\al}}{L^2_{\rm V}(t)}d\vol_{g(t)} \cr
&& \leq \frac{1}{4}I_2 +
(1-\al) \int_{_{Z_t} }  \frac{\Sc_{g(t)}^2 }{ ({\rm V}-t)^{1+\al} L^2_{\rm V}(t)}d\vol_{g(t)}\cr
&& \leq  \frac{1}{4}I_2 +  ({\rm V}-t)^{-1-\al}
( \int_{Z_t}  |\Sc_{g(t)}|^{2+\be}d\vol_{g(t)})^{\frac{2}{2+\be}}
m(t)^{\frac{\be}{2+\be}} \cr
&& \leq   \frac{1}{4}I_2+ ({\rm V}-t)^{-1-\al}    ( \int_{Z_t} |\Sc_{g(t)}|^{2+\be}d\vol_{g(t)})^{\frac{2}{2+\be} +  \frac{\be}{2+\be}} ({\rm V}-t)^{ \be(1-2\al) }\cr
&&= \ \frac{1}{4}I_2+ ({\rm V}-t)^{-1-\al + \be -2\be\al}      \int_{Z_t} |\Sc_{g(t)}|^{2+\be}d\vol_{g(t)} \cr
&& = \frac{1}{4}I_2+    ({\rm V}-t)^{-1+\al}({\rm V}-t)^{2\alpha(1- 4\alpha)} \int_{Z_t}  |\Sc_{g(t)}|^{2+\be}d\vol_{g(t)} \cr
&& \leq \frac{1}{4}I_2+    ({\rm V}-t)^{-1+\al} T  \int_{Z_t}  |\Sc_{g(t)}|^{2+\be}d\vol_{g(t)}, 
\end{eqnarray}
where we set $\beta =  4\al$ and use $\al < \frac{1}{4}$ in the last equality, and we use $L_{\rm V} \geq 1$ in deriving the second inequality and $T \geq 1$  in the last inequality.
 
For $f(t):= \int_N \frac{|\Rc_{g(t)}|_{g(t)}^2}{L_{\rm V}(t)} d\vol_{g(t)}$ we have shown: 
\begin{eqnarray}
 && \partt f(t)  \cr
 &&    \leq  \hat c + I_1 -\frac{3}{4}I_2+    \frac{2 T}{({\rm V}-t)^{1-\al}}   f(t)      
 +  \frac{T}{({\rm V}-t)^{1-\al}}   \int_N |\Sc_{g(t)}|^{2 +4\al} d\vol_{g(t)}.
 \end{eqnarray}

  Taking the derivative in time of $ e^{b({\rm V}-t)^{\al}} f(t)$  
 for $ b \geq    \frac{10 T}{\al}  $ we see that this implies
 
\begin{eqnarray}
&& \partt e^{b({\rm V}-t)^{\al}}f(t)  \cr
 && = -e^{b({\rm V}-t)^{\al}} \al b({\rm V}-t)^{\al -1} f(t)  +  e^{b({\rm V}-t)^{\al}}  \partt f(t) \cr
 &&  \leq -e^{b({\rm V}-t)^{\al}} \al b({\rm V}-t)^{\al -1} f(t) \cr
 &&  \ \ \  +  e^{b({\rm V}-t)^{\al}} \Big(\hat c +I_1 -\frac{3}{4}I_2 + ({\rm V}-t)^{\alpha -1}\frac{b \al}{2}f(t)
 + \frac{ T}{({\rm V}-t)^{1-\alpha}}\int_N|\Sc_{g(t)}|^{2+4\al} d\vol_{g(t)}\Big) \cr
 &&     \leq -e^{b({\rm V}-t)^{\al}} \frac{\al b}{2}({\rm V}-t)^{\al -1} f(t)\cr 
 &&  \ \ \  +  e^{b({\rm V}-t)^{\al}} \Big(\hat c +I_1 -\frac{3}{4}I_2 
 +  \frac{T}{({\rm V}-t)^{1-\alpha}}\int_N |\Sc_{g(t)}|^{2+4\al} d\vol_{g(t)}\Big)
 \end{eqnarray}
Integrating from $r$ to $s$ we obtain

\begin{eqnarray}
 && e^{b({\rm V}-s)^{\al}} \int_N \frac{|\Rc_{g(s)}|_{g(s)}^2}{L_{\rm V}(s)} d\vol_{g(s)} + \int_r^s  \int_N  
 \frac{\al b e^{b({\rm V}-t)^{\al}}}{2({\rm V}-t)^{1-\alpha} }  \frac{|\Rc_{g(t)}|_{g(t)}^2(t)}{L_{\rm V}(t)} d\vol_{g(t)}dt \cr
&& \ \ \ \ \ \ \ +    \frac{9}{8}\int_r^s  \int_N e^{b({\rm V}-t)^{\al}}  ({\rm V}-t)^{1-\alpha}\frac{|\Rc_{g(t)}|_{g(t)}^4(t)}{L^2_{\rm V}(t)}   d\vol_{g(t)}    dt \cr
    && \leq e^{b({\rm V}-r)^{\al}} \int_N \frac{|\Rc_{g(r)}|_{g(r)}^2}{L_{\rm V}(r)} d\vol_{g(r)}   + e^{b({\rm V}-r)^{\al}} \hat c (s-r) \cr
    &&   \ \ \  + 
  \int_r^s  \int_N  \frac{e^{b({\rm V}-t)^{\al}}  }{({\rm V}-t)^{1-\alpha}}  \Big(  8|\Rm_{g(t)}|_{g(t)}^2 +   T |\Sc_{g(t)}|^{2+4\alpha} \Big) d\vol_{g(t)} dt 
\end{eqnarray} 
as required.

\end{proof}

One of the terms on the right hand side of the above formula  involves  the 
full Riemannian curvature tensor. 
We will see in the following, that this may be replaced by a term only involving weighted integrals of the scalar curvature in the case that a)  the solution to the Ricci flow we are considering  
  is a four dimensional real solution, or
b) the solution to the Ricci flow we are considering is a K\"ahler-Ricci flow on a closed  manifold in any dimension.
  The key ingredient to do this will be:
  
  i) 
  the four dimensional  Chern-Gauss-Bonnet theorem on a smooth manifold $N$ with boundary in case a) and 

 ii)  Theorem \ref{KaehlerIntEst} of this paper in case b). 
 
  Theorem  \ref{KaehlerIntEst} enables us   
 to compare the $L^2$ integral of the full curvature tensor at time $t$ with the $L^2$ integral of the scalar curvature at time $t$, the initial values, and various quantities on the  boundary  of $N,$ which may be estimated in the case that the metric 
 is well controlled on a compact neighbourhood of the boundary of 
$N$.

    Theorem \ref{KaehlerIntEst} may be  seen as a local Ricci flow  version  of the well known 
integral formulae  involving the first and second Chern classes as stated in Section E, Chapter 2 in  \cite{Besse}.


\begin{thm}\label{firstmainwie}
For $n\in\N,$ $1\leq {\rm V} < T < \infty$ 
let   $(M^{n},g(t))_{t\in [0,T)},$  be a smooth,  real  solution
to Ricci flow, $\partt g(t) = -2\Rc_{g(t)}.$  
We assume that $n=4$ or that $(M^n,g_0)$ is Kähler and closed ( complex dimension $m= \frac n 2$, real dimension $n$ ) .
Assume further that 
 $\Sc_{g(t)} \geq -1$ for all $t\in [0,T),$   and  
 $\al \in (0,\frac{1}{12}),$  $N  \subseteq M,$ $\Omega$   and $(M^{n},g(t))_{t\in [0,T)}$    are  as in ({\bf B}). 
Then there exists a constant $\hat c_0 = \hat c_0(N,g_0,T,\Omega,   g|_{\Omega} ) < \infty$  such that for 
$b \geq \hat b(\al,T) =  \frac{2000 T}{\al}$ we have    for all $r<s$ with $ r,s \in (0,{\rm V}] $ that
\begin{eqnarray}\label{RIntEst} 
 &&  \int_N \frac{|\Rc_{g(s)}|_{g(s)}^2}{(2{\rm V}  + \Sc_{g(s)}({\rm V}-s)^{1-\al}} d\vol_{g(s)} \cr
 &&   +    \frac{1}{2}\int_r^s  \int_N   ({\rm V}-t)^{1-\alpha}\frac{|\Rc_{g(t)}|_{g(t)}^4}{ ( 2{\rm V}  + \Sc_{g(t)}({\rm V}-t)^{1-\al}   )^2}   d\vol_{g(t)}    dt \cr
   && \leq  e^{b({\rm V}-r)^{\al}} \frac{1}{\al}\hat c_0 (s-r)^{\al}  + 
  e^{b({\rm V}-r)^{\al}} \int_N \frac{|\Rc_{g(r)}|_{g(r)}^2 }{ 2{\rm V}  + \Sc_{g(r)}({\rm V}-r)^{1-\al} } d\vol_{g(r)}   \cr
    &&    + e^{b({\rm V}-r)^{\al}}     \int_r^s  \int_N  \frac{1 }{({\rm V}-t)^{1-\alpha}}      |\Sc_{g(t)}|^{2+12\alpha}(t) d\vol_{g(t)}    dt .
\end{eqnarray}

\end{thm}
\begin{proof}
In both the real four dimensional and K\"ahler case, we must estimate the term $\int_N |\Rm_{g(t)}|_{g(t)}^2 dV_{g(t)} $ appearing on the right side of the formula \eqref{FirstRIntEst}. 
We begin with the  real  four dimensional case.  
 The generalised Chern-Gauss-Bonnet Theorem, see for example 
 \cite{Gil}, or \cite{Troy} Sec 4.4,   
 tells us that 
 \begin{eqnarray} 
 && \int_N |\Rm_{g(t)}|_{g(t)}^2 d\vol_g(t) = 
  \int_N  4|\Rc_{g(t)}|_{g(t)}^2 d\vol_g(t)   -  \int_N \Sc_{g(t)}^2d\vol_g(t) \cr
&& \ \ \ \ +  8\pi^2 \chi(N) + c_2(N;II_{\boundary  N \subseteq B_{g(t)}(N,1)},g(t)|_{\Omega})  \cr
&& \leq \int_N  4|\Rc_{g(t)}|_{g(t)}^2  +\hat c
\end{eqnarray}
for all $t\in [0,T)$ where $\hat c< \infty $ is a constant depending only on $(N,g_0), \Omega, g|_{\Omega}, T,$ in view of ({\bf B}). 
Note that this formula and conclusion are still correct if $N$ is not orientated and $\chi(N):= \frac{1}{2}\chi(\ti N)$ 
where $\chi(\ti{N})$ is the Euler characteristic of the {\it double cover of $N$}, which is oriented (see Theorem  15.41 of \cite{Lee}).

In the Kähler case, we remember that there is a K\"ahler-Ricci flow solution
having the complex metric $\hat g_0$ obtained naturally from $g_0$ as
its initial value.    Uniqueness of solutions in the closed  case implies that the real solution $(M,g(t))_{t\in [0,T)}$ is the real solution obtained by taking the corresponding real solution which is naturally obtained  from the K\"ahler solution.  In particular,  Theorem \ref{KaehlerIntEst} is valid. Using Theorem \ref{KaehlerIntEst} in place of the generalised Chern-Gauss-Bonnet  formula, it is possible     to obtain an estimate similar to the one above for the solution $(M,g(t))_{t\in [0,T)}:$ For $m=\frac n 2, $ 
\newpage 
\begin{eqnarray}
&& \int_N |\Rm_{g(t)}|_{g(t)}^2(t)d\vol_{g(t)}\cr 
&& =\int_N \Sc_{g(t)}^2d\vol_{g(t)}+\int_N(|\Rm_{g_0}|^2_{g_0}-\Sc^2_{g_0})d\vol_{g_0}\cr
&&  \ \ \   + C(t)   -\frac{1}{(m-2)!}\int_N\sum^{m-3}_{j=0}C^{j}_{m-2} \rho^2_{g_0}\wedge\omega_0^j\wedge(-t\rho_0)^{m-2-j} \cr
&&  \ \ \  -\frac{1}{(m-2)!}\int_N\sum^{m-3}_{j=0}C^{j}_{m-2} (\rho_{g_0}^2-2c_2(\omega_0))\wedge\omega_0^j\wedge(-t\rho_0))^{m-2-j} \cr
&&\leq \int_N \Sc_{g(t)}^2(t)d\vol_{g(t)}+\hat c
\end{eqnarray}
where  
 we used Theorem \ref{KaehlerIntEstCor}  in the last step to estimate
$|C(t)|$, and 
$\hat c< \infty $ is once again a constant depending only on   on $(N,g_0), \Omega, g|_{\Omega}, T.$

In either case we have:
\begin{eqnarray}
\int_N |\Rm_{g(t)}|_{g(t)}^2 dV_{g(t)}\leq \int_N 4|\Rc_{g(t)}|_{g(t)}^2dV_{g(t)} + \int_N \Sc_{g(t)}^2 dV_{g(t)}  +\hat c \label{into}
\end{eqnarray}
for all $t\in [0,T)$.
Now we use this to estimate the term on the right hand side of \eqref{FirstRIntEst}
involving the $L^2$ norm of the full Riemannian curvature tensor.

\begin{eqnarray}
&& \int_N \frac{|\Rm_{g(t)}|_{g(t)}^2}{({\rm V}-t)^{1-\al}}  dV_{g(t)}\cr
&& \leq   4\int_N \frac{|\Rc_{g(t)}|_{g(t)}^2}{({\rm V}-t)^{1-\al}} dV_{g(t)} +  \int_N \frac{\Sc_{g(t)}^2}{({\rm V}-t)^{1-\al}} dV_{g(t)} +
 \frac{\hat c}{({\rm V}-t)^{1-\al}}  \cr
&& =   4 \int_N \frac{|\Rc_{g(t)}|_{g(t)}^2L_{\rm V}(t)}{L_{\rm V}(t)({\rm V}-t)^{1-\al}} dV_{g(t)}  +  \int_N \frac{\Sc_{g(t)}^2}{({\rm V}-t)^{1-\al}}dV_{g(t)}  +
 \frac{\hat c}{({\rm V}-t)^{1-\al}} \cr
&& = 8 {\rm V}\int_N \frac{ |\Rc_{g(t)}|_{g(t)}^2}{L_{\rm V}(t)({\rm V}-t)^{1-\al}} dV_{g(t)} +
4   \int_N \frac{|\Rc_{g(t)}|_{g(t)}^2 \Sc_{g(t)}}{L_{\rm V}(t)} dV_{g(t)} \cr 
&& \ \ \ +    \int_N \frac{\Sc_{g(t)}^2}{({\rm V}-t)^{1-\al}}dV_{g(t)}   +   \frac{\hat c}{({\rm V}-t)^{1-\al}} \cr
&& \leq  \frac{  8T }{({\rm V}-t)^{1-\al} } \int_N \frac{|\Rc_{g(t)}|_{g(t)}^2}{L_{\rm V}(t)}dV_{g(t)}+
\int_N \frac{|\Rc_{g(t)}|_{g(t)}^4 ({\rm V}-t)^{1-\al}}{40L^2_{\rm V}(t)}dV_{g(t)}\cr
&& \ \ \  +  \frac{200}{({\rm V}-t)^{1-\al}}
\int_N \Sc_{g(t)}^2dV_{g(t)}  +   
 \frac{\hat c}{({\rm V}-t)^{1-\al}}.  
\end{eqnarray}
Hence 
\begin{eqnarray}
 &&\int_r^s  \int_N  \frac{e^{b({\rm V}-t)^{\al}}  }{({\rm V}-t)^{1-\alpha}}   8|\Rm_{g(t)}|_{g(t)}^2 dV_{g(t)}dt\cr
&& \leq  \int_r^s \int_N \frac{64 T e^{b({\rm V}-t)^{\al}}  }{({\rm V}-t)^{1-\alpha}}   
  \frac{|\Rc_{g(t)}|_{g(t)}^2}{L_{\rm V}(t)} dV_{g(t)}dt\cr
  && \ \ \ +
\int_r^s  \int_N e^{b({\rm V}-t)^{\al}} \frac{|\Rc_{g(t)}|_{g(t)}^4 ({\rm V}-t)^{1-\al}}{4L^2_{\rm V}(t)} dV_{g(t)}dt\cr
&& \ \ \  + \int_r^s \int_N  e^{b({\rm V}-t)^{\al}} \frac{2000}{({\rm V}-t)^{1-\al}}
\Sc_{g(t)}^2  dV_{g(t)}dt +   
 \int_r^s \frac{\hat c e^{b({\rm V}-t)^{\al}}}{({\rm V}-t)^{1-\al}} dt. 
\end{eqnarray}
Inserting this into \eqref{FirstRIntEst}, we are able to    absorb the   terms 
$\int_r^s \int_N \frac{64 T e^{b({\rm V}-t)^{\al}}  }{({\rm V}-t)^{1-\alpha}}   
  \frac{|\Rc_{g(t)}|_{g(t)}^2}{L_{\rm V}(t)} dV_{g(t)}dt $\\
   and $
\int_r^s  \int_N  e^{b({\rm V}-t)^{\al}} \frac{|\Rc_{g(t)}|_{g(t)}^4 ({\rm V}-t)^{1-\al}}{4L^2_{\rm V}(t)}dV_{g(t)}dt$ by the terms appearing on the left hand side of \eqref{FirstRIntEst}, if we assume that  $b\al \geq 2000 T$  and remember that $T> {\rm V} \geq 1$.  In doing so we obtain

\begin{eqnarray}
 && e^{b({\rm V}-s)^{\al}} \int_N \frac{|\Rc_{g(s)}|_{g(s)}^2}{L_{\rm V}(s)} d\vol_{g(s)} 
 +    \int_r^s  \int_N e^{b({\rm V}-t)^{\al}}  ({\rm V}-t)^{1-\alpha}\frac{|\Rc_{g(t)}|_{g(t)}^4(t)}{2L_{\rm V}^2(t)}   d\vol_{g(t)}    dt \cr
    && \leq e^{b({\rm V}-r)^{\al}} \int_N \frac{|\Rc_{g(r)}|_{g(r)}^2(r)}{L_{\rm V}(r)} d\vol_{g(r)}   + e^{b({\rm V}-r)^{\al}} \frac{1}{\al}\hat c (s-r)^{\al}  \cr
    &&    + 
  e^{b({\rm V}-r)^{\al}} \int_r^s  \int_N  \frac{1  }{({\rm V}-t)^{1-\alpha}} \Big(  2000\Sc_{g(t)}^2  +  T |\Sc_{g(t)}|^{2+4\alpha}\Big) d\vol_{g(t)} dt , \label{theo}
\end{eqnarray} 
where we estimated $$|{\rm V}-r|^{\al} -|{\rm V}-s|^{\al}    = 
|{\rm V}-s + (s-r)|^{\al}- |{\rm V}-s|^{\al} \leq |{\rm V}-s|^{\al} + |s-r|^{\al}- |{\rm V}-s|^{\al}  = |s-r|^{\al}$$ for $0< \al <1,$  and  $e^{b({\rm V}-t)^{\al}} \leq e^{b({\rm V}-r)^{\al}}$ for all $ t\in [r,s].$
Using that $\Sc_{g(t)} \geq -1$, we see 
 $\partt {\rm Vol}_{g(t)}(N) = \partt (\int_N d\vol_{g(t)})  = -\int_N \Sc_{g(t)} d\vol_{g(t)} \leq  d\vol_{g(t)} ( N )$.
This means that  ${\rm Vol}_{g(t)}(N) \leq e^{T}{\rm Vol}_{g(0)}(N)   =:  C_0(T)$ and consequently $\int_N (2000 \Sc_{g(t)}^{ 2} + T|\Sc_{g(t)}|^{ 2 + 4\alpha} )  d\vol_{g(t)} \leq  \int_N |\Sc_{g(t)}|_{g(t)}^{ 2 + 12\alpha} d\vol_{g(t)} + c(\al,T) {\rm Vol}_{g(t)} (N) 
\leq  \int_N \Sc_{g(t)}^{ 2 + 12\alpha} d\vol_{g(t)}  + C_0(T). $
Using this and    $e^{b({\rm V}-s)^{\al}} \geq 1$  in  \eqref{theo}  above implies the result.
\end{proof}

If one further assumes above that the scalar curvature is bounded in the $L^{2+v}$ sense for some $v>0$ these estimates may be improved and used to   obtain further estimates.

\begin{thm}\label{secondmainwei}
For $n\in\N,$ $1\leq {\rm V} < T < \infty,$ 
let   $(M^{n},g(t))_{t\in [0,T)},$  be a smooth,  real  solution
to Ricci flow, $\partt g(t) = -2\Rc_{g(t)}.$  
We assume that $n=4$ or that $(M^n,g_0)$ is Kähler and closed ( complex dimension $m= \frac n 2$, real dimension $n$ ).
Assume further that 
 $\Sc_{g(t)} \geq -1$ for all $t\in [0,T),$   and  
 $\al \in (0,\frac{1}{12}),$  $N  \subseteq M,$ $\Omega$   and $(M^{n},g(t))_{t\in [0,T)}$    are  as in ({\bf B}) and that

$$\int_N |\Sc_{g(t)}|_{g(t)}^{2+12\alpha}(t) d\vol_{g(t)}  \leq C_0< \infty$$ for all $t\in [0,T).$ 
Then, for all $l  \in [0,T)$ we have  
\begin{eqnarray} 
&& i)   \int_{N} |\Rm_{g(l)}|_{g(l)}^2 d\vol_{g(l)}  
  \leq \hat c_1\label{Riemestgood} \\\nonumber
   && ii)   \int_{ {\rm V}-2s}^{{\rm V}-s} \int_{N \cap B_{g(t)}(p,r) } |\Rc_{g(t)}|_{g(t)}^4 d\vol_{g(t)} dt  \\
   && \ \ \    \leq \hat c_1 \frac{1}{s^{1-\alpha}}  + 
  \hat c_1 \sup \{     |\Rc_{g(t)}|_{g(t)}^2 \ | \   t \in [{\rm V}-2s,{\rm V}-s] , x    \in B_{g(t)}(p,r)    \}s^{1+\alpha  }
  \end{eqnarray}
for all $r\in (0,\infty),$ if  $s \in (0,1)$ and ${\rm V}-2s>0$, where  $\hat c_1$ is a constant depending on $ n,N,g(0), \Omega, g|_{\Omega},  \frac{1}{\al}, T, C_0.$ 
\end{thm}

\begin{proof}
We first show  (i)  in the case  $l\leq 1$. 
 Since $|\Rm_{g(t)}|_{g(t)}\leq 1 $ for $t\in [0,1],$ we see that $\int_N |\Rm_{g(l)}|_{g(l)} d\vol_{g(l)} 
 \leq {\rm Vol}_{g(l)}(N) \leq e^2 {\rm Vol}_{g(0)}(N)$ and hence (i) holds in that case.
 
We now show (i) for $l\geq 1$.
 Let ${\rm V} \in [1,T)$. 
 Remembering that $L_{\rm V}(t) = 2{\rm V} + \Sc_{g(t)}({\rm V}-t)^{1-\alpha}$ and in particular that $L_{\rm V}({\rm V})= 2{\rm V}$, we see by 
 choosing  $r=0$ and $s ={\rm V} <  T$  in the estimate  \eqref{RIntEst}, that 
\begin{eqnarray}
&&  \int_N \frac{|\Rc_{g({\rm V})}|_{g({\rm V})}^2}{2{\rm V}} d\vol_{g({\rm V})}  \cr
&& = 
  \int_N \frac{|\Rc_{g(t)}|_{g(t)}^2}{L_{\rm V}({\rm V})} d\vol_{g({\rm V})}  \cr
  && 
   \leq     \hat a_1(n,N,g(0),\hat c_0,\frac{1}{\al}, T, C_0) \cr
&&    =  \hat a_1(n,N,g(0), \Omega, g|_{\Omega},  \frac{1}{\al}, T, C_0)< \infty, \label{needit}
\end{eqnarray}
In the following we denote any constant of the type $\hat a_1(n,N,g(0), \Omega, g|_{\Omega},  \frac{1}{\al}, T, C_0)$
simply be $ \hat c$ and such constants are assumed to be larger than one (if not add one to it). 
 For example the constant  $\hat c^2100$  will also be denoted by $\hat c$. 
Using inequality \eqref{into} and \eqref{needit}, we get 
$ \int_N |\Rm_{g(\rm V)}|_{g(\rm V)}^2   d\vol_{g({\rm V})}  \leq \hat c $ for arbitrary ${\rm V} \in [1,T),$ which implies (i) for $l\in[1,T)$.
Hence (i) is correct for $l\in [0,T)$.
 
 We now show (ii) for ${\rm V}\leq 1$. We have 
 \begin{eqnarray}
  &&  \int_{{\rm V}-2s}^{{\rm V}-s} \int_{N \cap B_{g(t)}(p,r)} |\Rc_{g(t)}|_{g(t)}^4 d\vol_{g(t)}dt  \cr
  &&   \leq   \Big(\int_{{\rm V}-2s}^{{\rm V}-s} \int_{N \cap B_{g(t)}(p,r)} |\Rc_{g(t)}|_{g(t)}^2 d\vol_{g(t)}dt \Big)\cr 
  && \ \ \ \ \cdot   \Big( \sup \{|\Rc_{g(t)}|_{g(t)}^2 \ | \ x \in B_{g(t)}(p,r), t \in [{\rm V}-2s,{\rm V}-s] \} \Big) \cr
  &&  \leq  s e^2 {\rm Vol}_{g(0)}(N)    \sup \{|\Rc_{g(t)}|_{g(t)}^2 \ | \ x \in B_{g(t)}(p,r), t \in [{\rm V}-2s,{\rm V}-s] \}   \cr
  && \leq e^2 {\rm Vol}_{g(0)}(N)   ( \sup \{|\Rc_{g(t)}|_{g(t)}^2\ | \ x \in B_{g(t)}(p,r), t \in [{\rm V}-2s,{\rm V}-s] \}) (s^{1+\alpha} + s^{1-\alpha})  \cr
  &&  \leq e^2 {\rm Vol}_{g(0)}(N)     \sup \{|\Rc_{g(t)}|_{g(t)}^2\ | \ x \in B_{g(t)}(p,r), t \in [{\rm V}-2s,{\rm V}-s] \} s^{1+\alpha} \cr
  && \ \ \  + e^2 {\rm Vol}_{g(0)}(N)   s^{1-\alpha} 
    \end{eqnarray} 
    and hence (ii) holds for ${\rm V} \leq 1$.\\
We now show (ii) for ${\rm V}\geq 1$.
Choosing $r=0$  and letting $s={\rm V}< T$ in the estimate  \eqref{RIntEst},  we also get
\begin{eqnarray}
  && \int_0^{\rm V} \int_N \frac{|\Rc_{g(t)}|_{g(t)}^4 ({\rm V}-t)^{1-\al} }{L_{\rm V}(t)^2} d\vol_{g(t)}dt   \leq   \hat  c\label{starstar}
  \end{eqnarray} 
 We  define time dependent sets $V_t$, $\Omega_t$    by \\
\begin{eqnarray}
V_t :=  \{ x \in N \ |  \ \Sc_{g(t)}({\rm V}-t)^{1-\al} \leq 1\},  \\
 \Omega_t := \{ x \in N \ | \ \Sc_{g(t)}({\rm V}-t)^{1-\al} \geq 1\}. 
 \end{eqnarray}
  
On $\Omega_t$ we  have for $m(t) = {\rm Vol}_{g(t)}(\Omega_t)$ that 
$\frac{1}{({\rm V}-t)^{2-2\al}}m(t)  \leq \int_{\Omega_t} \Sc_{g(t)}^2dV_{g(t)} \leq C_0
$ and hence 
\begin{eqnarray}
m(t) \leq ({\rm V}-t)^{2-2\al}C_0
\end{eqnarray}

  Using this fact, and the inequality \eqref{starstar}, we calculate  for $|{\rm V}-S|\leq 1,$  and 
   $v =12\al$, 
 \begin{eqnarray}
&& \int_{S}^{\rm V} \int_{\Omega_t} |\Rc_{g(t)}|_{g(t)}^{2}  d\vol_{g(t)}  dt \cr 
&& =  \int_{S}^{\rm V} \int_{\Omega_t} |\Rc_{g(t)}|_{g(t)}^{2}  \cdot  \Big(\frac{({\rm V}-t)^{1-\al}}{(\Sc_{g(t)}({\rm V}-t)^{1-\al})^2 } \Big)^{\frac{1}{2}}
 \cdot  \Big(\frac{(\Sc_{g(t)}({\rm V}-t)^{1-\al})^2 }{({\rm V}-t)^{1-\al}} \Big)^{\frac{1}{2}}d\vol_{g(t)}  dt \cr 
 && \leq   \Big(\int_{S}^{\rm V} \int_{\Omega_t} \frac{|\Rc_{g(t)}|_{g(t)}^{4}({\rm V}-t)^{1-\al}}{(\Sc_{g(t)}({\rm V}-t)^{1-\al})^2}  d\vol_{g(t)} dt \Big)^{\frac{1}{2} }\cdot \Big( \int_{S}^{\rm V} \int_{\Omega_t}    \frac{(\Sc_{g(t)}({\rm V}-t)^{1-\al})^2}{({\rm V}-t)^{1-\al}}         d\vol_{g(t)}  dt\Big)^{ \frac{1}{2}} \cr 
 && =    \Big(\int_{S}^{\rm V} \int_{\Omega_t} \frac{|\Rc_{g(t)}|_{g(t)}^{4}({\rm V}-t)^{1-\al}}{(\frac 1 2 \Sc_{g(t)}({\rm V}-t)^{1-\al}+\frac 1 2 \Sc_{g(t)}({\rm V}-t)^{1-\al})^2}  d\vol_{g(t)} dt\Big)^{\frac{1}{2} }\cr 
 && \ \ \ \cdot \Big( \int_{S}^{\rm V} \int_{\Omega_t}   \frac{(\Sc_{g(t)}({\rm V}-t)^{1-\al})^2}{({\rm V}-t)^{1-\al}}   d\vol_{g(t)}  dt\Big)^{ \frac{1}{2}} \cr 
 && \leq   \Big(\int_{S}^{\rm V} \int_{\Omega_t} \frac{|\Rc_{g(t)}|_{g(t)}^{4}({\rm V}-t)^{1-\al}}{(\frac 1 2  +\frac 1 4 \Sc_{g(t)}({\rm V}-t)^{1-\al})^2}  d\vol_{g(t)} dt\Big)^{\frac{1}{2} }\cdot\Big( \int_{S}^{\rm V} \int_{\Omega_t}   (\frac{(\Sc_{g(t)}({\rm V}-t)^{1-\al})^2}{({\rm V}-t)^{1-\al}} )        d\vol_{g(t)}  dt\Big)^{ \frac{1}{2}}  \cr 
 && \leq 4 \Big(\int_{S}^{\rm V} \int_{\Omega_t} \frac{|\Rc_{g(t)}|_{g(t)}^{4}({\rm V}-t)^{1-\al}}{(2 +  \Sc_{g(t)}({\rm V}-t)^{1-\al})^2}  d\vol_{g(t)} dt\Big)^{\frac{1}{2} }\cdot\Big( \int_{S}^{\rm V} \int_{\Omega_t}   (\frac{(\Sc_{g(t)}({\rm V}-t)^{1-\al})^2}{({\rm V}-t)^{1-\al}} )        d\vol_{g(t)}  dt\Big)^{ \frac{1}{2}}  \cr 
 && =   4{\rm V} \Big(\int_{S}^{\rm V} \int_{\Omega_t} \frac{|\Rc_{g(t)}|_{g(t)}^{4}({\rm V}-t)^{1-\al}}{(2{\rm V} +  {\rm V}\Sc_{g(t)}({\rm V}-t)^{1-\al})^2}  d\vol_{g(t)} dt\Big)^{\frac{1}{2} }\cdot\Big( \int_{S}^{\rm V} \int_{\Omega_t}   \frac{(\Sc_{g(t)}({\rm V}-t)^{1-\al})^2}{({\rm V}-t)^{1-\al}}         d\vol_{g(t)}  dt\Big)^{ \frac{1}{2}}  \cr 
 &&  \leq 4{\rm V} \Big(\int_{S}^{\rm V} \int_{\Omega_t} \frac{|\Rc_{g(t)}|_{g(t)}^{4}({\rm V}-t)^{1-\al}}{(2{\rm V} +  \Sc_{g(t)}({\rm V}-t)^{1-\al})^2}  d\vol_{g(t)} dt\Big)^{\frac{1}{2} }\cdot\Big( \int_{S}^{\rm V} \int_{\Omega_t}  \frac{(\Sc_{g(t)}({\rm V}-t)^{1-\al})^2}{({\rm V}-t)^{1-\al}}         d\vol_{g(t)}  dt\Big)^{ \frac{1}{2}} \cr 
  && =4{\rm V}\Big(  \int_{S}^{\rm V} \int_{\Omega_t} \frac{|\Rc_{g(t)}|_{g(t)}^{4}({\rm V}-t)^{1-\al}}{L_{\rm V}(t)^2}  d\vol_{g(t)} dt\Big)^{\frac{1}{2} } \cdot   \Big( \int_{S}^{\rm V} \int_{\Omega_t}   \Sc_{g(t)}^2 ({\rm V}-t)^{1-\al}       d\vol_{g(t)}  dt\Big)^{ \frac{1}{2}} \cr 
&& \leq \hat c    \Big( \int_{S}^{\rm V} ({\rm V}-t)^{1-\al  } \int_{\Omega_t}  \Sc_{g(t)}^2      d\vol_{g(t)}  dt\Big)^{\frac{1}{2}}  \cr 
  && \leq   \hat c     \Big( \int_{S}^{\rm V} ({\rm V}-t)^{1-\al  } (\int_{\Omega_t} \Sc_{g(t)}^{2 (1+\frac{v}{2})}  d\vol_{g(t)} )^{\frac{1}{(1+\frac{v}{2})}} m(t)^{\frac{v}{2(1+\frac{v}{2})}}       dt \Big)^{ \frac{1}{2}}  \cr
 && \leq    \hat c     C_0    ( \int_{S}^{\rm V} ({\rm V}-t)^{1-\al }   ({\rm V}-t)^{\frac{v}{4}}dt )^{ \frac{1}{2}} \cr
 && =      \hat c    (({\rm V}-S)^{2 -\alpha + \frac{12\al}{4}} )^{ \frac{1}{2}} \cr 
 && = \hat c      ({\rm V}-S)^{1+  \al }  \label{star}
\end{eqnarray}
where we use $\frac 1 2 \Sc_{g(t)}({\rm V}-t)^{1-\al} \geq    \frac{1}{4}\Sc_{g(t)}({\rm V}-t)^{1-\al}$ on $\Omega_t $ in the second inequality, and we use that $|{\rm V}-t| \leq 1$ for $t\in [S,{\rm V}]$ since $|{\rm V}-S|\leq 1$. Hence 
\begin{eqnarray}
 &&    \int_{{\rm V}-2s}^{{\rm V}-s}  \int_{N \cap 
 \Omega_t \cap B_{g(t)}(p,r) }|\Rc_{g(t)}|_{g(t)}^4 d\vol_{g(t)} dt  \cr 
 && 
    \leq   \sup \{ |\Rc_{g(t)}|_{g(t)}^2 \ | \ t \in [{\rm V}-2s,{\rm V}-s ] ,x  \in B_{g(t)}(p,r) \cap N\}\cr
 &&   \ \ \  \cdot \Big(  \int_{{\rm V}-2s}^{{\rm V}-s} 
 \int_{N \cap  \Omega_t \cap B_{g(t)}(p,r) }|\Rc_{g(t)}|_{g(t)}^2 d\vol_{g(t)} dt \Big)  \cr
  &&  \leq  \sup \{ |\Rc_{g(t)}|_{g(t)}^2 \ | \ t \in [{\rm V}-2s,{\rm V}-s ]  ,x  \in B_{g(t)}(p,r) \cap N\} 
  \cr
  && \ \ \ \cdot  \Big( \int_{{\rm V}-2s}^{{\rm V}-s} 
   \int_{  \Omega_t   }|\Rc_{g(t)}|_{g(t)}^2 d\vol_{g(t)} dt  \Big) \cr
  &&  \leq \hat c \sup \{ |\Rc_{g(t)}|_{g(t)}^2\ | \ t \in [{\rm V}-2s,{\rm V}-s ]  ,x  \in B_{g(t)}(p,r) \cap N\} s^{1+\al }  \label{ineqy1}
 \end{eqnarray}
 in view of \eqref{star} with $S= {\rm V}-2s$.

Using    inequality  \eqref{starstar} again, we see 
\begin{eqnarray}\label{inbetw}
&&  \int_{{\rm V}-2s}^{{\rm V}} \int_{N \cap B_{g(t)}(p,r) } \frac{|\Rc_{g(t)}|_{g(t)}^4({\rm V}-t)^{1-\alpha} }{L^2_{\rm V}(t)} d\vol_{g(t)} dt  \leq \hat c.
\end{eqnarray} 
 Using this inequality,  $ 0\leq s \leq |{\rm V}-t|$ for $ t \in ({\rm V}-2s,{\rm V}-s)$,$ {\rm V}\geq 1$ and   $L_{\rm V}\leq 4{\rm V}$ on $V_t$ when ${\rm V}\geq 1$, we obtain   
\begin{eqnarray}
&&
 \int_{{\rm V}-2s}^{{\rm V}-s} \int_{N \cap B_{g(t)}(p,r) \cap V_t} |\Rc_{g(t)}|_{g(t)}^4 s^{1-\alpha}  d\vol_{g(t)} dt \cr 
 &&   \leq   100{\rm V}^2 \int_{{\rm V}-2s}^{{\rm V}-s} \int_{N \cap B_{g(t)}(p,r)  \cap V_t}\frac{|\Rc_{g(t)}|_{g(t)}^4({\rm V}-t)^{1-\alpha}}{100{\rm V}^2}  d\vol_{g(t)} dt \cr
&& \leq 100{\rm V}^2 \int_{{\rm V}-2s}^{{\rm V} } \int_{N \cap B_{g(t)}(p,r)  \cap V_t}\frac{|\Rc_{g(t)}|_{g(t)}^4({\rm V}-t)^{1-\alpha}}{L^2_{\rm V}(t)}  d\vol_{g(t)} dt \cr
&& \leq 
  \hat c,
\end{eqnarray}
and hence 
 \begin{eqnarray}
&&
 \int_{{\rm V}-2s}^{{\rm V}-s}   \int_{N \cap B_{g(t)}(p,r)  \cap V_t}  |\Rc_{g(t)}|_{g(t)}^4   d\vol_{g(t)} dt \cr
&& \leq \frac{\hat c}{s^{1-\alpha} }. \label{ineqy2} 
\end{eqnarray}

The inequalities \eqref{ineqy1} and \eqref{ineqy2} show us that 
\begin{eqnarray}
 &&    \int_{{\rm V}-2s}^{{\rm V}-s} \int_{N  \cap B_{g(t)}(p,r) }|\Rc_{g(t)}|_{g(t)}^4 d\vol_{g(t)} dt \cr
   && =   \int_{{\rm V}-2s}^{{\rm V}-s} \int_{N \cap  B_{g(t)}(p,r)  \cap V_t}|\Rc_{g(t)}|_{g(t)}^4 d\vol_{g(t)} dt \cr
   && \ \ +  \int_{{\rm V}-2s}^{{\rm V}-s} \int_{N \cap    B_{g(t)}(p,r)  \cap \Omega_t } |\Rc_{g(t)}|_{g(t)}^4 d\vol_{g(t)} dt   \cr
   &&  \leq \frac{\hat c}{s^{1-\alpha}} + \hat c \sup \{ |\Rc_{g(t)}|_{g(t)}^2\ | \ t \in [V-2s,V-s] , x \in B_{g(t)}(p,r) \cap N\} s^{1+ \alpha} .
   \end{eqnarray}
     Hence, (ii) holds for ${\rm V}\geq 1.$ 
 
    \end{proof}

  Under a further {\it weak non-inflating} assumption,  namely that 
${\rm Vol}_{g(t)}(B_{g(t)}(p,\sqrt{V-t})) \leq \si_1 |V-t|^2$   for some constant $\si_1< \infty$ and some fixed $p\in N,$    for     all $t\in [V-1 ,V),$   for some $V \in [1,T],$  
we obtain further   local integral estimates. This is not the usual non-inflating condition when $n> 4$,  and is a weaker condition in that case: The usual non-inflating condition would be
${\rm Vol}_{g(t)}(B_{g(t)}(p,r)) \leq \si r^n$ for all $r\leq 1$ and this implies 
${\rm Vol}_{g(t)}(B_{g(t)}(p,\sqrt{V-t})   ) \leq \si_1 |V-t|^{n/2} \leq \si_1 |V-t|^2$ when
$|V-t|\leq 1$.

\begin{thm}\label{AnotherInt}
For $n\in\N,$ $1\leq {\rm V} \leq  T < \infty,$ 
let   $(M^{n},g(t))_{t\in [0,T)},$  be a smooth,  real  solution
to Ricci flow, $\partt g(t) = -2\Rc_{g(t)}.$  
We assume that $n=4$ or that $(M^n,g_0)$ is Kähler and closed ( complex dimension $m= \frac n 2$, real dimension $n$ ).
Assume further that 
 $\Sc_{g(t)} \geq -1$ for all $t\in [0,T),$   and  
 $\al \in (0,\frac{1}{12}),$  $N  \subseteq M,$ $\Omega$   and $(M^{n},g(t))_{t\in [0,T)}$    are  as in ({\bf B}) and that

$$\int_N |\Sc_{g(t)}|_{g(t)}^{2+12\alpha}(t) d\vol_{g(t)}  \leq C_0< \infty$$ for all $t\in [0,T)$ 
and  that there exist  a $\si_1>0,$ $p\in N,$ 
  such  that  
\begin{eqnarray}
{\rm Vol}_{g(t)}(B_{g(t)}(p, \sqrt{V-t} ) ) \leq \si_1 |V-t|^2  \label{NonExpandT}
\end{eqnarray}
   for     all $t\in [V-1,V)$ (as we noted  in Remark  \ref{noninflrem}, this   non-inflating assumption is not necessary in the case that $M$ is closed, due to the work of Bamler \cite{Bam}).
Then we have
\begin{eqnarray}
&& \int_{S}^V \int_{B_{g(t)}(p,\sqrt{V-t})} |\Rc_{g(t)}|_{g(t)}^{2+ \si}  d\vol_{g(t)}  dt \cr
&& \leq   \hat c_2    (V-S)^{1 + \frac{\al}{16}} \label{Riccestgood}
 \end{eqnarray} 
   for all $\sigma>0$ sufficiently small   ($\si \leq \al^3$ suffices) and  $S \in [V-1,V),$  
 where $\hat c_2< \infty$ depends on $\si_1, n,N,g(0), \Omega, g|_{\Omega}, \al , T, C_0 .$   
   For $V= T,$ we mean  
 \begin{eqnarray}
  \lim_{V \upto T} \int_{S}^V \int_{B_{g(t)}(p,\sqrt{T-t})} |\Rc_{g(t)}|_{g(t)}^{2+ \si}   d\vol_{g(t)}  dt \leq  \hat c_2   (T-S)^{1 + \frac{\al}{16}} \label{Riccestgood2}
 \end{eqnarray}
  
\end{thm}

 \begin{proof}

 Constants which only depend at most on $  \si_1, n,N,g(0), \Omega, g|_{\Omega},  \frac{1}{\al}, T, C_0, g|_{N\times [0,1]}$ shall simply be denoted by $\hat c$.    The constant may change from line to line, but will still be denoted by $\hat c$. 
 
If $\dist_{g(V-1)}(p,\boundary N) \leq 5$ then $\dist_{g(t)}(p,\boundary N) \leq 9 $ for all $t\in [V-1,V]$  due to the condition ({\bf B}), and so $B_{g(t)}(p,\sqrt{V-t}) \subseteq \Omega$, where $\Omega $ comes from ({\bf B}) and the curvature is bounded by $1$	on $\Omega$, and hence the result holds.

If $\dist_{g(V-1)}(p,\boundary N) \geq 5$  then , due to the condition ({\bf B}), we have
$\dist_{g(t)}(p,\boundary N) >1$ for all $t\in [V-1,V]$ and so   $B_{g(t)}(p,\sqrt{V-t}) \subseteq N$ for all $t\in [V-1,V].$ 
 
Now , we   use a  similar calculation to \eqref{star}  combined with the non-inflating estimate.
 For $S>0$ with $|V-S|\leq 1,$ using Hölder's inequality and  equation \eqref{RIntEst}, we get
\begin{eqnarray}
&& \int_{S}^{\rm V} \int_{B_{g(t)}(p,\sqrt{{\rm V}-t})} |\Rc_{g(t)}|_{g(t)}^{2+\si}  d\vol_{g(t)}  dt \cr
&& =  \int_{S}^{\rm V}\int_{B_{g(t)}(p,\sqrt{{\rm V}-t})}   |\Rc_{g(t)}|_{g(t)}^{2+\si}  \cdot \Big(\frac{({\rm V}-t)^{1-\al}}{L^2_{\rm V}(t)}  \Big)^{\frac{2+\si}{4}}
 \cdot   \Big(\frac{L^2_{\rm V}(t)}{({\rm V}-t)^{1-\al}}  \Big)^{\frac{2+\si}{4}}d\vol_{g(t)}  dt \cr
 && \leq    \Big(\int_{S}^{\rm V} \int_{B_{g(t)}(p,\sqrt{{\rm V}-t})} \frac{|\Rc_{g(t)}|_{g(t)}^{4}({\rm V}-t)^{1-\al}}{L^2_{\rm V}(t)}  d\vol_{g(t)} dt \Big)^{\frac{2+\si}{4} }\cr
 && \ \ \ \ \ \ \ \ \ \ \ \cdot    \Big( \int_{S}^{\rm V} \int_{B_{g(t)}(p,\sqrt{{\rm V}-t})}   \Big(\frac{L^2_{\rm V}(t)}{({\rm V}-t)^{1-\al}} \Big)^{\frac{(2+\si)}{(2-\si)}}        d\vol_{g(t)}  dt \Big)^{ \frac{2-\si}{4}}\cr
&& \leq \hat c  \Big( \int_{S}^{\rm V} \int_{B_{g(t)}(p,\sqrt{{\rm V}-t})}   \Big(\frac{L^2_{\rm V}(t)}{({\rm V}-t)^{1-\al}} \Big)^{\frac{(2+\si)}{(2-\si)}}        d\vol_{g(t)}  dt\Big)^{ \frac{2-\si}{4}}   \cr
&&  \leq \hat c   \Big( \int_{S}^{\rm V} \int_{B_{g(t)}(p,\sqrt{{\rm V}-t})}   \Big(\frac{ {\rm V}^2  + \Sc_{g(t)}^2({\rm V}-t)^{2(1-\al)}}{({\rm V}-t)^{1-\al}} \Big)^{\frac{(2+\si)}{(2-\si)}}        d\vol_{g(t)}  dt\Big)^{ \frac{2-\si}{4}}   \cr
&& =    \hat c   \Big( \int_{S}^{\rm V} \int_{B_{g(t)}(p,\sqrt{{\rm V}-t})}   \Big(\frac{ {\rm V}^2 }{({\rm V}-t)^{1-\al}}
 + \Sc_{g(t)}^2 ({\rm V}-t)^{1-\al} \Big)^{\frac{(2+\si)}{(2-\si)}}        d\vol_{g(t)}  dt \Big)^{ \frac{2-\si}{4}}.  \label{itty}
 \end{eqnarray} 
We  write $     \frac{(2+\si)(1-\al)}{(2-\si)} = 1 -\al + \ep_{\si}$  where  $\ep_{\si}= \frac{2\si(1-\al)}{2-\si}   \to 0 \ \ {\rm as  }\ \  \si \to 0, $
 and  $ \ \frac{2(2+\si)}{(2-\si)} = 2+\be_{\si}$ where $\be_{\si}  = \frac{4\si}{2-\si} \to 0 \ \ {\rm as  }\ \  \si   \to 0. $
 
We choose $v>0$ such that $(1+v)(2+\beta_{\si}) = 2+ 12\alpha,$ for $\si>0$ sufficiently small 
 : that is we set $v:= \frac{2+ 12\alpha}{2+\be_{\si}}-1
 =  \frac{  {12}\alpha-\be_\si}{2+\be_{\si} }
 \in   ( \frac{{11}\alpha}{2}, {6} \alpha),$   { for $\si >0$ sufficiently small ($\si \leq \al^3$ suffices),}  
 and hence 
 $\frac{2v}{1+v} \geq  4\alpha.$

 Then { we have}
 \begin{eqnarray}  
&& \int_{S}^{\rm V} \int_{B_{g(t)}(p,\sqrt{{\rm V}-t})} |\Rc_{g(t)}|_{g(t)}^{2+\si}  d\vol_{g(t)}  dt \cr 
&& \leq  \hat c   \Big( \int_{S}^{\rm V} \int_{B_{g(t)}(p,\sqrt{{\rm V}-t})}   \Big(\frac{ {\rm V}^2 }{({\rm V}-t)^{1-\al}}
 + \Sc_{g(t)}^2 ({\rm V}-t)^{1-\al} \Big)^{\frac{(2+\si)}{(2-\si)}}        d\vol_{g(t)}  dt \Big)^{ \frac{2-\si}{4}} \cr
 && \leq  \hat c  \Big( \int_{S}^{\rm V} \int_{B_{g(t)}(p,\sqrt{{\rm V}-t})}    \frac{ {\rm V}^4 }{({\rm V}-t)^{1-\al +\ep_{\si} }}
 + |\Sc_{g(t)}|^{2+\be_{\si}} ({\rm V}-t)^{1-\al +\ep_{\si} }          d\vol_{g(t)}  dt  \Big)^{ \frac{2-\si}{4}}  \cr 
 && \leq   \hat c  \Big[  \int_{S}^{\rm V} {\rm V}^4  \si_1 ({\rm V}-t)^{1+\al- \ep_{\si} }dt  \cr 
 && \ \ \ \ \ \ + \int_S^{\rm V}  ({\rm V}-t)^{1-\al + \ep_{\si}} \int_{B_{g(t)}(p,\sqrt{{\rm V}-t})} |\Sc_{g(t)}|^{2+\be_{\si}}  d\vol_{g(t)}  dt\Big]^{ \frac{2-\si}{4}}\cr 
 && \leq  \hat c \Big[    ({\rm V}-S)^{2+\al - \ep_{\si}}   \cr 
&& \ \ \ \ \  \ \   + \int_S^{\rm V}  ({\rm V}-t)^{1-\al+\ep_{\si} } \Big(\int_{B_{g(t)}(p,\sqrt{{\rm V}-t})} |\Sc_{g(t)}|^{2+{12}\alpha}  d\vol_{g(t)} \Big)^{\frac{1}{(1+v)}}  (\si_1 |{\rm V}-t|^2)^{ \frac{v }{1+v}}dt \Big]^{ \frac{2-\si}{4}}\cr 
 && \leq  \hat c \Big(  ( {\rm V}-S)^{2+ \al -\ep_{\si}}    + (\si_1)^{\frac v{1+v} }   \hat c_2 \int_S^{\rm V} ({\rm V}-t)^{1-\al + \ep_{\si} +   \frac{2v}{1+v}} dt   \Big)^{ \frac{2-\si}{4}}\cr 
 && \leq  \hat c  \Big(    ( {\rm V}-S)^{2+\frac{\al}{ 2}} +  ({\rm V}-S)^{2+\frac{\alpha}{2}-\frac{3\alpha}{2}+\ep_{\si} + \frac{2v }{1+v}}\Big)^{ \frac{2-\si}{4}} \cr 
  && \leq  \hat c  \Big(    ( {\rm V}-S)^{2+\frac{\al}{ 2}} +  ({\rm V}-S)^{2+\frac{\alpha}{2}}\Big)^{ \frac{2-\si}{4}} \cr 
 && \leq  \hat c   ({\rm V}-S)^{1 + \frac{\al}{16}}.
 \end{eqnarray}

    Letting ${\rm V} \upto T$  and using Fatou's Lemma  gives us the required estimate. 
 \end{proof}

\section{Local integral estimates for the K\"ahler-Ricci flow}\label{sec:5}
In this section, we derive local integral estimates, in certain cases,  for the K\"ahler-Ricci flow $(\ref{Kahler})$
\begin{eqnarray}\label{Kahler}
\begin{cases}
  \frac{\partial\omega(t)}{\partial t}=-\Ric_{\omega(t)},\\
  \\
  \omega(t)|_{t=0}=\omega_0
  \end{cases}
 \end{eqnarray}
where $\omega_0$ is a smooth K\"ahler metric on  a closed manifold.

In  local complex coordinates $(z_1,\cdots,z_n)$, we write the K\"ahler form $\omega$, the Ricci form $\rho_g$ and form $c_2(\omega)$ for a given fixed Kähler manifold as follows.
\begin{equation*}\label{09003}
\begin{split}
&\ \omega= \frac{\sqrt{-1}}{2\pi}g_{i\bar{j}}dz^i\wedge d\bar{z}^j,\ \ \ \rho_{g}=\frac{\sqrt{-1}}{2\pi}\Sc_{i\bar{j}}dz^i\wedge d\bar{z}^j,\\
&\ \Omega^i_j=\frac{\sqrt{-1}}{2\pi}g^{i\bar{p}}\Sc_{j\bar{p}k\bar{l}}dz^k\wedge d\bar{z}^l,\ \ \ c_2(\omega)=\frac{1}{2}\sum^{n}_{i,j=1}(\Omega^i_i \wedge\Omega^j_j-\Omega^i_j\wedge\Omega^j_i).
\end{split}
\end{equation*}
In fact, $c_{2}(\omega)$ is a real closed $(2,2)$-form which represents the second Chern class $c_2(M)$ and $\rho_g=\sum\limits_{i=1}^n\Omega^i_i$.

We recall further  that the the following formulae for closed Kähler manifolds  hold  in any dimension. 
 This was first observed essentially by Apte \cite{APTE}. A detailed proof can be found in Zheng's work \cite{ZHENG}. 
\begin{lem}\label{400}(Apte \cite{APTE}, Zheng \cite{ZHENG}) Let $(M,\omega_g)$ be a smooth closed  K\"ahler manifold with complex dimension $n$. Then  
\begin{equation}\label{401a}
c_1^2(M)\cdot[\omega_g]^{n-2}=(n-2)!\int_M (\Sc_g^2-|\Ric_g|_g^2)\ dV_g,
\end{equation}
\begin{equation}\label{401b}
c_2(M)\cdot[\omega_g]^{n-2}=\frac{(n-2)!}{2}\int_M (\Sc_g^2-2|\Ric_g|_g^2+|\Rm_g|_g^2)\ dV_g
\end{equation}
\begin{equation}\label{401c}
(c_2(M)-c^2_1(M))\cdot[\omega_g]^{n-2}=\frac{(n-2)!}{2}\int_M (|\Rm_g|_g^2-\Sc_g^2)\ dV_g.
\end{equation}
\end{lem}
 The following lemma is the pointwise version of Lemma \ref{400}, and the proof thereof has been included in  Appendix A, for the readers convenience. 
\begin{lem}\label{09001}(Apte \cite{APTE}, Zheng \cite{ZHENG}) Let $(M,\omega_g)$ be a smooth K\"ahler manifold with complex dimension $n$. Then
\begin{equation*}\label{09002}
\begin{split}
\rho_{g}\wedge \rho_{g}\wedge \omega^{n-2}&=\frac{1}{n(n-1)}(\Sc^2_{g}-|\Ric_{g}|_{g}^2)\ \omega^n,\\
c_2(\omega)\wedge \omega^{n-2}&=\frac{1}{2n(n-1)}(\Sc^2_{g}-2|\Ric_{g}|_{g}^2+|\Rm_g|_{g}^2)\ \omega^n,\\
(\rho_g\wedge \rho_g-2c_2(\omega))\wedge\omega^{n-2}&=\frac{1}{n(n-1)}(|\Ric_{g}|_{g}^2-|\Rm_g|_{g}^2)\ \omega^n. 
\end{split}
\end{equation*}
\end{lem}
Since $\rho_g$ is a real closed $(1,1)$-form, by the following $dd^{c}$-Lemma, we deduce that there exists a real $(1,1)$-form $\alpha(t)$ such that
\begin{equation}\label{09004}
\rho_g\wedge \rho_g-2c_2(\omega)=\rho_{g_0}\wedge \rho_{g_0}-2c_2(\omega_0)+\sqrt{-1}\partial\bar{\partial}\alpha(t).
\end{equation}
Here $d^c=\sqrt{-1}(\bar{\partial}-\partial)$, and hence $dd^c=2\sqrt{-1}\partial\bar{\partial}$. The $dd^{c}$-Lemma we use  here is from Deligne-Griffiths-Morgan-Sullivan's paper \cite{DGMS}.
\begin{lem}\label{09006} (Deligne-Griffiths-Morgan-Sullivan, $dd^{c}$-Lemma, Lemma $(5.11)$ in \cite{DGMS}) Let $(M,\omega_g)$ be a smooth closed K\"ahler manifold.   If $\psi$ is a differential form such that $d\psi=0$ and $d^c\psi=0$, and such that $\psi=d\gamma$, then $\psi=dd^c\alpha$ for some $\alpha$.

\end{lem}

   { Denote $g=g(t)$ and $\omega=\omega(t)$. The change in time of the metric and the K\"ahler form satisfy  the identities 
\begin{equation}\label{0900700000}
\begin{split}
&\rho_{g}=\rho_{g_0}+\sqrt{-1}\partial\bar{\partial}f(t),\\
&\omega_t = \omega_0 - t \rho_0 + \sqrt{-1}\partial\bar{\partial}\varphi(t)\\
&  \ \ \  =  \omega_0  + \Omega_t \\
\end{split}
\end{equation}
where \begin{eqnarray}\label{rickeqns}
&&f(t)=\log\frac{\omega_0^n}{\omega^n} \cr
&& \varphi(t) = \int_0^t \log\Big( \frac{\omega^n(s)}{\omega_0^n}\Big)  ds \cr
&&  \Omega_t := -t\rho_0   +  \sqrt{-1}\partial\bar{\partial}\varphi(t).
\end{eqnarray}
$\Omega_t$ is a $(1,1)$ form, but we do not claim that it belongs to the  K\"ahler class of a metric.}

\begin{lem}\label{09029} Let $(M,\omega_0)$ be a smooth closed K\"ahler manifold with complex dimension $n$. Then along the K\"ahler-Ricci flow (\ref{Kahler}), we have the following formulae.
\begin{equation*}\label{09008}
\begin{split}
\frac{1}{n(n-1)}(|\Ric_{g}|_{g}^2-\Sc_{g}^2)\ \omega^n&=\frac{1}{n(n-1)}(|\Ric_{g_0}|_{g_0}^2-\Sc^2_{g_0})\ \omega_0^n-\sqrt{-1}\partial\bar{\partial}f(t)\wedge\sqrt{-1}\partial\bar{\partial}f(t)\wedge\omega^{n-2}\\
&-2\sqrt{-1}\partial\bar{\partial}f(t)\wedge \rho_{g_0}\wedge\omega^{n-2}-\sum^{n-3}_{j=0}C^{j}_{n-2} \rho^2_{g_0}\wedge\omega_0^j\wedge\Omega_t^{n-2-j},
\end{split}
\end{equation*}
and
\begin{equation}\label{09009}
\begin{split}
\frac{1}{n(n-1)}(|\Ric_{g}|_{g}^2-|\Rm_g|^2_{g})\ \omega^n&=\frac{1}{n(n-1)}(|\Ric_{g_0}|_{g_0}^2-|\Rm_{g_0}|^2_{g_0})\ \omega_0^n+\sqrt{-1}\partial\bar{\partial}\alpha(t)\wedge\omega^{n-2}\\
&\ \ \ +\sum^{n-3}_{j=0}C^{j}_{n-2}(\rho_{g_0}^2-2c_2(\omega_0))\wedge\omega_0^j\wedge\Omega_t^{n-2-j}.
\end{split}
\end{equation}
\end{lem}
\begin{proof} 
By using Lemma \ref{09001} and directly computing, we have
\begin{equation}\label{09010}
\begin{split}
&\ \ \ \frac{1}{n(n-1)}(|\Ric_{g}|_{g}^2-\Sc_{g}^2)\ \omega^n=-\rho_g\wedge\rho_g\wedge\omega^{n-2}\\
&=-(\rho_{g_0}+\sqrt{-1}\partial\bar{\partial}f(t))\wedge(\rho_{g_0}+\sqrt{-1}\partial\bar{\partial}f(t))\wedge\omega^{n-2}\\
&=-\rho_{g_0}\wedge\rho_{g_0}\wedge\omega^{n-2}-2\sqrt{-1}\partial\bar{\partial}f(t)\wedge\rho_{g_0}\wedge\omega^{n-2}-\sqrt{-1}\partial\bar{\partial}f(t)\wedge\sqrt{-1}\partial\bar{\partial}f(t)\wedge\omega^{n-2}\\
&=-\rho_{g_0}\wedge\rho_{g_0}\wedge(\omega_0+\Omega_t)^{n-2}-2\sqrt{-1}\partial\bar{\partial}f(t)\wedge\rho_{g_0}\wedge\omega^{n-2}\\
&\ \ \ -\sqrt{-1}\partial\bar{\partial}f(t)\wedge\sqrt{-1}\partial\bar{\partial}f(t)\wedge\omega^{n-2}\\
&=- \rho_{g_0}\wedge\rho_{g_0}\wedge \omega_0^{n-2}-\sum^{n-3}_{j=0} C^j_{n-2}\rho_{g_0}\wedge\rho_{g_0}\wedge\omega^j_0\wedge\Omega_t^{n-2-j}\\
&\ \ \ -2\sqrt{-1}\partial\bar{\partial}f(t)\wedge\rho_{g_0}\wedge\omega^{n-2}-\sqrt{-1}\partial\bar{\partial}f(t)\wedge\sqrt{-1}\partial\bar{\partial}f(t)\wedge\omega^{n-2}\\
&=\frac{1}{n(n-1)}(|\Ric_{g_0}|_{g_0}^2-\Sc_{g_0}^2)\ \omega_0^n-\sum^{n-3}_{j=0} C^j_{n-2} \rho_{g_0}\wedge\rho_{g_0}\wedge\omega^j_0\wedge\Omega_t^{n-2-j}\\
&\ \ \ -2\sqrt{-1}\partial\bar{\partial}f(t)\wedge\rho_{g_0}\wedge\omega^{n-2}-\sqrt{-1}\partial\bar{\partial}f(t)\wedge\sqrt{-1}\partial\bar{\partial}f(t)\wedge\omega^{n-2}.
\end{split}
\end{equation}
Using Lemma \ref{09001} and equality \eqref{09004}, we have
\begin{equation}\label{09011}
\begin{split}
&\ \ \ \frac{1}{n(n-1)}(|\Ric_{g}|_{g}^2-|\Rm_g|_{g}^2)\ \omega^n=(\rho_g\wedge \rho_g-2c_2(\omega))\wedge\omega^{n-2}\\
&=(\rho_{g_0}\wedge \rho_{g_0}-2c_2(\omega_0))\wedge\omega^{n-2}+\sqrt{-1}\partial\bar{\partial}\alpha(t)\wedge\omega^{n-2}\\
&=(\rho_{g_0}\wedge \rho_{g_0}-2c_2(\omega_0))\wedge(\omega_0+\Omega_t)^{n-2}+\sqrt{-1}\partial\bar{\partial}\alpha(t)\wedge\omega^{n-2}\\
&=\sum^{n-3}_{j=0}C^j_{n-2} (\rho_{g_0}\wedge \rho_{g_0}-2c_2(\omega_0))\wedge\omega^j_0\wedge\Omega_t^{n-2-j}\\
&\ \ \ + (\rho_{g_0}\wedge \rho_{g_0}-2c_2(\omega_0))\wedge \omega_0^{n-2}+\sqrt{-1}\partial\bar{\partial}\alpha(t)\wedge\omega^{n-2}\\
&=\frac{1}{n(n-1)}(|\Ric_{g_0}|_{g_0}^2-|\Rm_{g_0}|_{g_0}^2)\ \omega_0^n+\sqrt{-1}\partial\bar{\partial}\alpha(t)\wedge\omega^{n-2}\\
&\ \ \ +\sum^{n-3}_{j=0}C^j_{n-2} (\rho_{g_0}\wedge \rho_{g_0}-2c_2(\omega_0))\wedge\omega^j_0\wedge\Omega_t^{n-2-j}.
\end{split}
\end{equation}
which  completes the proof.
\end{proof}
We have explicit formulae for the terms $f(t)$ and $\varphi(t)$ which appear in the right hand side of the  above formula. 
In order to better understand the term $\partial \bar{\partial} \alpha(t)$,which also appears there,    we derive a local formula :

\begin{lem}\label{09030}
In local coordinates $(z^1,\cdots, z^n)$, we have
\begin{equation}\label{09012}
\begin{split}
\sqrt{-1}\partial\bar{\partial}\alpha(t)&=\frac{1}{8\pi^2}\bar{\partial}\Big((\Gamma^i_{jk}-\tilde{\Gamma}^i_{jk})\bar{\partial}\big(\Gamma^j_{i\gamma}dz^k\wedge dz^\gamma\big)\Big)-\frac{1}{8\pi^2}\bar{\partial}\Big((\tilde{\Gamma}^i_{jk}-\Gamma^i_{jk})\bar{\partial}\big(\tilde{\Gamma}^j_{i\gamma}dz^k\wedge dz^\gamma\big)\Big)\\
&\ \ +\frac{1}{8\pi^2}\partial\Big((\Gamma^{\bar{i}}_{\bar{j}\bar{k}}-\tilde{\Gamma}^{\bar{i}}_{\bar{j}\bar{k}})\partial\big(\Gamma^{\bar{j}}_{\bar{i}\bar{\gamma}}d\bar{z}^k\wedge d\bar{z}^\gamma\big)\Big)-\frac{1}{8\pi^2}\partial\Big((\tilde{\Gamma}^{\bar{i}}_{\bar{j}\bar{k}}-\Gamma^{\bar{i}}_{\bar{j}\bar{k}})\partial\big(\tilde{\Gamma}^{\bar{j}}_{\bar{i}\bar{\gamma}}d\bar{z}^k\wedge d\bar{z}^\gamma\big)\Big),
\end{split}
\end{equation}
where $\Gamma$ and $\tilde{\Gamma}$ are Christoffel symbols with respect to metrics $g$ and $g_0$ respectively.
  Furthermore,
\begin{eqnarray}
\sqrt{-1}\partial\bar{\partial}\alpha(t) \wedge \frac{\omega^{n-2}}{(n-2)!} 
= d\beta(t)   \wedge \frac{\omega^{n-2}}{(n-2)!} 
\end{eqnarray}
where  
\begin{eqnarray}
&&  \beta(t) :=  \frac{1}{8\pi^2} \Big(
(\Gamma^i_{jk}-\tilde{\Gamma}^i_{jk}) (\Sc^{\ j}_{i\  \gamma  \bar r} +\ti{\Sc}^{\ j}_{i\ \gamma \bar r} ) dz^{\bar r} \wedge dz^k\wedge dz^\gamma  \Big) \cr
&&\ \ \ \ \ \ \ \   +\frac{1}{8 \pi^2}\Big( (\Gamma^{\bar{i}}_{\bar{j}\bar{k}}-\tilde{\Gamma}^{\bar{i}}_{\bar{j}\bar{k}})
( \Sc^{\ \bar{j}}_{\bar{i}\ \bar{\gamma} r}   +  \tilde{\Sc}^{\ \bar{j}}_{\bar{i}\ \bar{\gamma} r} )  dz^r \wedge d\bar{z}^k\wedge d\bar{z}^\gamma  \Big),  \label{secclaim} 
\end{eqnarray}
is a well defined (coordinate independent) tensor.
 
\end{lem}
\begin{proof}
Direct calculations show that
\begin{equation}\label{09013}
\begin{split}
\Omega^i_j\wedge\Omega^j_i&=(\frac{\sqrt{-1}}{2\pi}g^{i\bar{p}}\Sc_{j\bar{p}k\bar{l}}dz^k\wedge d\bar{z}^l)\wedge(\frac{\sqrt{-1}}{2\pi}g^{j\bar{q}}\Sc_{i\bar{q}\gamma\bar{\delta}}dz^\gamma\wedge d\bar{z}^\delta)\\
&=-\frac{1}{4\pi^2}(\Sc^{\ \ i}_{j\ \ k\bar{l}}dz^k\wedge d\bar{z}^l)\wedge(\Sc^{\ \ j}_{i\ \ \gamma\bar{\delta}}dz^\gamma\wedge d\bar{z}^\delta)\\
&=-\frac{1}{4\pi^2}(-\frac{\partial\Gamma^i_{jk}}{\partial\bar{z}^l}dz^k\wedge d\bar{z}^l)\wedge(-\frac{\partial\Gamma^j_{i\gamma}}{\partial\bar{z}^\delta}dz^\gamma\wedge d\bar{z}^\delta)\\
&=\frac{1}{4\pi^2}\bar{\partial}\Big(\Gamma^i_{jk}\bar{\partial}\big(\Gamma^j_{i\gamma}dz^k\wedge dz^\gamma\big)\Big)\\
&=\frac{1}{4\pi^2}\bar{\partial}\Big((\Gamma^i_{jk}-\tilde{\Gamma}^i_{jk})\bar{\partial}\big(\Gamma^j_{i\gamma}dz^k\wedge dz^\gamma\big)\Big)+\frac{1}{4\pi^2}\bar{\partial}\Big(\tilde{\Gamma}^i_{jk}\bar{\partial}\big(\Gamma^j_{i\gamma}dz^k\wedge dz^\gamma\big)\Big).
\end{split}
\end{equation}
The same computation  performed for the metric and K\"ahler form at time zero gives us 
\begin{equation}\label{09014}
\tilde{\Omega}^i_j\wedge\tilde{\Omega}^j_i=\frac{1}{4\pi^2}\bar{\partial}\Big((\tilde{\Gamma}^i_{jk}-\Gamma^i_{jk})\bar{\partial}\big(\tilde{\Gamma}^j_{i\gamma}dz^k\wedge dz^\gamma\big)\Big)+\frac{1}{4\pi^2}\bar{\partial}\Big(\Gamma^i_{jk}\bar{\partial}\big(\tilde{\Gamma}^j_{i\gamma}dz^k\wedge dz^\gamma\big)\Big).
\end{equation}
At the same time, we have 
\begin{equation*}\label{09016}
\begin{split}
\bar{\partial}\Big(\tilde{\Gamma}^i_{jk}\bar{\partial}\big(\Gamma^j_{i\gamma}dz^k\wedge dz^\gamma\big)\Big)&=(\bar{\partial}\tilde{\Gamma}^i_{jk})\wedge\bar{\partial}\big(\Gamma^j_{i\gamma}dz^k\wedge dz^\gamma\big)+\tilde{\Gamma}^i_{jk}\bar{\partial}^2\big(\Gamma^j_{i\gamma}dz^k\wedge dz^\gamma\big)\\
&=(\frac{\partial\tilde{\Gamma}^i_{jk}}{\partial \bar{z}^l}\sqrt{-1}dz^k\wedge d\bar{z}^l)\wedge(\frac{\partial\Gamma^j_{i\gamma}}{\partial \bar{z}^\delta}\sqrt{-1}dz^\gamma\wedge d\bar{z}^\delta)
\end{split}
\end{equation*}
and 
\begin{equation*}\label{09017}
\begin{split}
\bar{\partial}\Big(\Gamma^i_{jk}\bar{\partial}\big(\tilde{\Gamma}^j_{i\gamma}dz^k\wedge dz^\gamma\big)\Big)&=(\bar{\partial}\Gamma^i_{jk})\wedge\bar{\partial}\big(\tilde{\Gamma}^j_{i\gamma}dz^k\wedge dz^\gamma\big)+\Gamma^i_{jk}\bar{\partial}^2\big(\tilde{\Gamma}^j_{i\gamma}dz^k\wedge dz^\gamma\big)\\
&=(\frac{\partial\Gamma^i_{jk}}{\partial \bar{z}^\delta}\sqrt{-1}dz^k\wedge d\bar{z}^\delta)\wedge(\frac{\partial\tilde{\Gamma}^j_{i\gamma}}{\partial \bar{z}^l}\sqrt{-1}dz^\gamma\wedge d\bar{z}^l)\\
&=(\frac{\partial\Gamma^j_{ik}}{\partial \bar{z}^\delta}\sqrt{-1}dz^k\wedge d\bar{z}^\delta)\wedge(\frac{\partial\tilde{\Gamma}^i_{j\gamma}}{\partial \bar{z}^l}\sqrt{-1}dz^\gamma\wedge d\bar{z}^l)\\
&=(\frac{\partial\Gamma^j_{i\gamma}}{\partial \bar{z}^\delta}\sqrt{-1}dz^\gamma\wedge d\bar{z}^\delta)\wedge(\frac{\partial\tilde{\Gamma}^i_{jk}}{\partial \bar{z}^l}\sqrt{-1}dz^k\wedge d\bar{z}^l),
\end{split}
\end{equation*}
where we exchange $i$ and $j$ in the third equality, and $k$ and $\gamma$ in the last equality. These equalities imply that
\begin{equation}\label{09018}
\bar{\partial}\Big(\tilde{\Gamma}^i_{jk}\bar{\partial}\big(\Gamma^j_{i\gamma}dz^k\wedge dz^\gamma\big)\Big)=\bar{\partial}\Big(\Gamma^i_{jk}\bar{\partial}\big(\tilde{\Gamma}^j_{i\gamma}dz^k\wedge dz^\gamma\big)\Big).
\end{equation}
Conjugating, we get 
\begin{equation}\label{0613001}
\begin{split}
\Omega^{\bar{i}}_{\bar{j}}\wedge\Omega^{\bar{j}}_{\bar{i}}&=\frac{1}{4\pi^2}\partial\Big((\Gamma^{\bar{i}}_{\bar{j}\bar{k}}-\tilde{\Gamma}^{\bar{i}}_{\bar{j}\bar{k}})\partial\big(\Gamma^{\bar{j}}_{\bar{i}\bar{\gamma}}d\bar{z}^k\wedge d\bar{z}^\gamma\big)\Big)+\frac{1}{4\pi^2}\partial\Big(\tilde{\Gamma}^{\bar{i}}_{\bar{j}\bar{k}}\partial\big(\Gamma^{\bar{j}}_{\bar{i}\bar{\gamma}}d\bar{z}^k\wedge d\bar{z}^\gamma\big)\Big),\\
\tilde{\Omega}^{\bar{i}}_{\bar{j}}\wedge\tilde{\Omega}^{\bar{j}}_{\bar{i}}&=\frac{1}{4\pi^2}\partial\Big((\tilde{\Gamma}^{\bar{i}}_{\bar{j}\bar{k}}-\Gamma^{\bar{i}}_{\bar{j}\bar{k}})\partial\big(\tilde{\Gamma}^{\bar{j}}_{\bar{i}\bar{\gamma}}d\bar{z}^k\wedge d\bar{z}^\gamma\big)\Big)+\frac{1}{4\pi^2}\partial\Big(\Gamma^{\bar{i}}_{\bar{j}\bar{k}}\partial\big(\tilde{\Gamma}^{\bar{j}}_{\bar{i}\bar{\gamma}}d\bar{z}^k\wedge d\bar{z}^\gamma\big)\Big).
\end{split}
\end{equation}
We also have
\begin{equation}\label{0613002}
\partial\Big(\tilde{\Gamma}^{\bar{i}}_{\bar{j}\bar{k}}\partial\big(\Gamma^{\bar{j}}_{\bar{i}\bar{\gamma}}d\bar{z}^k\wedge d\bar{z}^\gamma\big)\Big)=\partial\Big(\Gamma^{\bar{i}}_{\bar{j}\bar{k}}\partial\big(\tilde{\Gamma}^{\bar{j}}_{\bar{i}\bar{\gamma}}d\bar{z}^k\wedge d\bar{z}^\gamma\big)\Big).
\end{equation}
Since $c_{2}(\omega)$ and $\rho_g=\sum_{i=1}^n\Omega^i_i$ are real forms, and
\begin{equation}
\Omega^i_j\wedge\Omega^j_i=\Omega^i_i \wedge\Omega^j_j-2c_2(\omega),
\end{equation}
all $\Omega^i_j\wedge\Omega^j_i$ $(i,j=1,\cdots, n)$ are real $(2,2)$-forms. By using $\eqref{09013}$, $\eqref{09014},$  and $\eqref{0613001}$, we have
\begin{equation*}\label{0613003}
\begin{split}
\Omega^i_j\wedge\Omega^j_i&=\frac{1}{2}\Omega^i_j\wedge\Omega^j_i+\frac{1}{2}\Omega^{\bar{i}}_{\bar{j}}\wedge\Omega^{\bar{j}}_{\bar{i}}\\
&=\frac{1}{8\pi^2}\bar{\partial}\Big((\Gamma^i_{jk}-\tilde{\Gamma}^i_{jk})\bar{\partial}\big(\Gamma^j_{i\gamma}dz^k\wedge dz^\gamma\big)\Big)+\frac{1}{8\pi^2}\bar{\partial}\Big(\tilde{\Gamma}^i_{jk}\bar{\partial}\big(\Gamma^j_{i\gamma}dz^k\wedge dz^\gamma\big)\Big)\\
&\ \ +\frac{1}{8\pi^2}\partial\Big((\Gamma^{\bar{i}}_{\bar{j}\bar{k}}-\tilde{\Gamma}^{\bar{i}}_{\bar{j}\bar{k}})\partial\big(\Gamma^{\bar{j}}_{\bar{i}\bar{\gamma}}d\bar{z}^k\wedge d\bar{z}^\gamma\big)\Big)+\frac{1}{8\pi^2}\partial\Big(\tilde{\Gamma}^{\bar{i}}_{\bar{j}\bar{k}}\partial\big(\Gamma^{\bar{j}}_{\bar{i}\bar{\gamma}}d\bar{z}^k\wedge d\bar{z}^\gamma\big)\Big),\\
\end{split}
\end{equation*}
and  
\begin{equation*}\label{0613004}
\begin{split}
\tilde{\Omega}^i_j\wedge\tilde{\Omega}^j_i&=\frac{1}{2}\tilde{\Omega}^i_j\wedge\tilde{\Omega}^j_i+\frac{1}{2}\tilde{\Omega}^{\bar{i}}_{\bar{j}}\wedge\tilde{\Omega}^{\bar{j}}_{\bar{i}}\\
&=\frac{1}{8\pi^2}\bar{\partial}\Big((\tilde{\Gamma}^i_{jk}-\Gamma^i_{jk})\bar{\partial}\big(\tilde{\Gamma}^j_{i\gamma}dz^k\wedge dz^\gamma\big)\Big)+\frac{1}{8\pi^2}\bar{\partial}\Big(\Gamma^i_{jk}\bar{\partial}\big(\tilde{\Gamma}^j_{i\gamma}dz^k\wedge dz^\gamma\big)\Big)\\
&\ \ +\frac{1}{8\pi^2}\partial\Big((\tilde{\Gamma}^{\bar{i}}_{\bar{j}\bar{k}}-\Gamma^{\bar{i}}_{\bar{j}\bar{k}})\partial\big(\tilde{\Gamma}^{\bar{j}}_{\bar{i}\bar{\gamma}}d\bar{z}^k\wedge d\bar{z}^\gamma\big)\Big)+\frac{1}{8\pi^2}\partial\Big(\Gamma^{\bar{i}}_{\bar{j}\bar{k}}\partial\big(\tilde{\Gamma}^{\bar{j}}_{\bar{i}\bar{\gamma}}d\bar{z}^k\wedge d\bar{z}^\gamma\big)\Big).
\end{split}
\end{equation*}
Locally, we have
\begin{equation}\label{090188}
\sqrt{-1\partial}\bar{\partial}\alpha(t)=\Omega^i_j\wedge\Omega^j_i-\tilde{\Omega}^i_j\wedge\tilde{\Omega}^j_i.
\end{equation}
Using $(\ref{09018})$ and $(\ref{0613002})$, we deduce
\begin{equation*}\label{09015}
\begin{split}
\sqrt{-1}\partial\bar{\partial}\alpha(t)&=\frac{1}{8\pi^2}\bar{\partial}\Big((\Gamma^i_{jk}-\tilde{\Gamma}^i_{jk})\bar{\partial}\big(\Gamma^j_{i\gamma}dz^k\wedge dz^\gamma\big)\Big)-\frac{1}{8\pi^2}\bar{\partial}\Big((\tilde{\Gamma}^i_{jk}-\Gamma^i_{jk})\bar{\partial}\big(\tilde{\Gamma}^j_{i\gamma}dz^k\wedge dz^\gamma\big)\Big)\\
&\ \ +\frac{1}{8\pi^2}\partial\Big((\Gamma^{\bar{i}}_{\bar{j}\bar{k}}-\tilde{\Gamma}^{\bar{i}}_{\bar{j}\bar{k}})\partial\big(\Gamma^{\bar{j}}_{\bar{i}\bar{\gamma}}d\bar{z}^k\wedge d\bar{z}^\gamma\big)\Big)-\frac{1}{8\pi^2}\partial\Big((\tilde{\Gamma}^{\bar{i}}_{\bar{j}\bar{k}}-\Gamma^{\bar{i}}_{\bar{j}\bar{k}})\partial\big(\tilde{\Gamma}^{\bar{j}}_{\bar{i}\bar{\gamma}}d\bar{z}^k\wedge d\bar{z}^\gamma\big)\Big)
\end{split}
\end{equation*}
 This proves the identity \eqref{09012}, which is  the first claim of the Lemma.
 
Since $\omega^{n-2}(t)$ is a $(n-2,n-2)$-form, $\big(\Gamma^i_{jk}-\tilde{\Gamma}^i_{jk})\bar{\partial}(\Gamma^j_{i\gamma}dz^k\wedge dz^\gamma\big)-\big(\tilde{\Gamma}^i_{jk}-\Gamma^i_{jk})\bar{\partial}(\tilde{\Gamma}^j_{i\gamma}dz^k\wedge dz^\gamma\big)$ is a $(2,1)$-form, and $(\Gamma^{\bar{i}}_{\bar{j}\bar{k}}-\tilde{\Gamma}^{\bar{i}}_{\bar{j}\bar{k}})\partial\big(\Gamma^{\bar{j}}_{\bar{i}\bar{\gamma}}d\bar{z}^k\wedge d\bar{z}^\gamma\big)-(\tilde{\Gamma}^{\bar{i}}_{\bar{j}\bar{k}}-\Gamma^{\bar{i}}_{\bar{j}\bar{k}})\partial\big(\tilde{\Gamma}^{\bar{j}}_{\bar{i}\bar{\gamma}}d\bar{z}^k\wedge d\bar{z}^\gamma\big)$ is a $(1,2)$-form, and we have
\begin{equation}
\begin{split}
\partial\Big(\big(\Gamma^i_{jk}-\tilde{\Gamma}^i_{jk})\bar{\partial}(\Gamma^j_{i\gamma}dz^k\wedge dz^\gamma\big)-\big(\tilde{\Gamma}^i_{jk}-\Gamma^i_{jk})\bar{\partial}(\tilde{\Gamma}^j_{i\gamma}dz^k\wedge dz^\gamma\big)\Big)\wedge\frac{\omega^{n-2}(t)}{(n-2)!}=0,\\
\bar{\partial}\Big((\Gamma^{\bar{i}}_{\bar{j}\bar{k}}-\tilde{\Gamma}^{\bar{i}}_{\bar{j}\bar{k}})\partial\big(\Gamma^{\bar{j}}_{\bar{i}\bar{\gamma}}d\bar{z}^k\wedge d\bar{z}^\gamma\big)-(\tilde{\Gamma}^{\bar{i}}_{\bar{j}\bar{k}}-\Gamma^{\bar{i}}_{\bar{j}\bar{k}})\partial\big(\tilde{\Gamma}^{\bar{j}}_{\bar{i}\bar{\gamma}}d\bar{z}^k\wedge d\bar{z}^\gamma\big)\Big)\wedge\frac{\omega^{n-2}(t)}{(n-2)!}=0.
\end{split}
\end{equation}
Hence 
\begin{equation}\label{09100}
\begin{split}
&\ \ \ \sqrt{-1}\partial\bar{\partial}\alpha(t)\wedge\frac{\omega^{n-2}(t)}{(n-2)!}\\
&=\frac{1}{8\pi^2}\bar{\partial}\Big((\Gamma^i_{jk}-\tilde{\Gamma}^i_{jk})\bar{\partial}\big(\Gamma^j_{i\gamma}dz^k\wedge dz^\gamma\big)-(\tilde{\Gamma}^i_{jk}-\Gamma^i_{jk})\bar{\partial}\big(\tilde{\Gamma}^j_{i\gamma}dz^k\wedge dz^\gamma\big)\Big)\wedge\frac{\omega^{n-2}(t)}{(n-2)!}\\
&\ \ +\frac{1}{8\pi^2}\partial\Big((\Gamma^{\bar{i}}_{\bar{j}\bar{k}}-\tilde{\Gamma}^{\bar{i}}_{\bar{j}\bar{k}})\partial\big(\Gamma^{\bar{j}}_{\bar{i}\bar{\gamma}}d\bar{z}^k\wedge d\bar{z}^\gamma\big)-(\tilde{\Gamma}^{\bar{i}}_{\bar{j}\bar{k}}-\Gamma^{\bar{i}}_{\bar{j}\bar{k}})\partial\big(\tilde{\Gamma}^{\bar{j}}_{\bar{i}\bar{\gamma}}d\bar{z}^k\wedge d\bar{z}^\gamma\big)\Big)\wedge\frac{\omega^{n-2}(t)}{(n-2)!}\\
&=\frac{1}{8\pi^2}d\Big((\Gamma^i_{jk}-\tilde{\Gamma}^i_{jk})\bar{\partial}\big(\Gamma^j_{i\gamma}dz^k\wedge dz^\gamma\big)-(\tilde{\Gamma}^i_{jk}-\Gamma^i_{jk})\bar{\partial}\big(\tilde{\Gamma}^j_{i\gamma}dz^k\wedge dz^\gamma\big)\Big)\wedge\frac{\omega^{n-2}(t)}{(n-2)!}\\
&\ \ +\frac{1}{8\pi^2}d\Big((\Gamma^{\bar{i}}_{\bar{j}\bar{k}}-\tilde{\Gamma}^{\bar{i}}_{\bar{j}\bar{k}})\partial\big(\Gamma^{\bar{j}}_{\bar{i}\bar{\gamma}}d\bar{z}^k\wedge d\bar{z}^\gamma\big)-(\tilde{\Gamma}^{\bar{i}}_{\bar{j}\bar{k}}-\Gamma^{\bar{i}}_{\bar{j}\bar{k}})\partial\big(\tilde{\Gamma}^{\bar{j}}_{\bar{i}\bar{\gamma}}d\bar{z}^k\wedge d\bar{z}^\gamma\big)\Big)\wedge\frac{\omega^{n-2}(t)}{(n-2)!}\\
&:=d\beta(t)\wedge\frac{\omega^{n-2}(t)}{(n-2)!},
\end{split}
\end{equation}
 where 
\begin{eqnarray}
&& \beta(t)=  \frac{1}{8\pi^2} \Big(
(\Gamma^i_{jk}-\tilde{\Gamma}^i_{jk})\bar{\partial}\big(    ( \Gamma^j_{i\gamma} + \tilde{\Gamma}^j_{i\gamma})   dz^k\wedge dz^\gamma\big) \Big) \cr
&& \ \ \ \ \ \  \ \   + \frac{1}{8 \pi^2}\Big((\Gamma^{\bar{i}}_{\bar{j}\bar{k}}-\tilde{\Gamma}^{\bar{i}}_{\bar{j}\bar{k}})
\partial\big(    ( \Gamma^{\bar{j}}_{\bar{i}\bar{\gamma}}   + \tilde{\Gamma}^{\bar{j}}_{\bar{i}\bar{\gamma}})    d\bar{z}^k\wedge d\bar{z}^\gamma\big) \Big),
\end{eqnarray} 
which may  be written  as 
\begin{eqnarray}
&&  \beta(t)= \frac{1}{8\pi^2} \Big(
(\Gamma^i_{jk}-\tilde{\Gamma}^i_{jk}) (\frac{\partial \Gamma^j_{i\gamma}}{\partial z^{\bar r} } + \frac{\partial {\tilde \Gamma}^j_{i\gamma}}{\partial z^{\bar r}})
 dz^{\bar r} \wedge dz^k\wedge dz^\gamma  \Big) \cr
&&\ \ \ \ \ \ \ \   +\frac{1}{8 \pi^2}\Big( (\Gamma^{\bar{i}}_{\bar{j}\bar{k}}-\tilde{\Gamma}^{\bar{i}}_{\bar{j}\bar{k}})
 (\frac{ \partial \Gamma^{\bar{j}}_{\bar{i}\bar{\gamma}}    } {\partial z^r}  +  \frac{ \partial    \tilde{\Gamma}^{\bar{j}}_{\bar{i}\bar{\gamma}}   }{\partial z^r} ).
    dz^r \wedge d\bar{z}^k\wedge d\bar{z}^\gamma  \Big) 
\end{eqnarray}
and hence 
\begin{eqnarray}
&&  \beta(t) = \frac{1}{8\pi^2} \Big(
(\Gamma^i_{jk}-\tilde{\Gamma}^i_{jk}) (\Sc^{\ j}_{i\  \gamma  \bar r} +\ti{\Sc}^{\ j}_{i\ \gamma \bar r} )  dz^{\bar r} \wedge dz^k\wedge dz^\gamma  \Big) \cr
&&\ \ \ \ \ \ \ \   +\frac{1}{8 \pi^2}\Big( (\Gamma^{\bar{i}}_{\bar{j}\bar{k}}-\tilde{\Gamma}^{\bar{i}}_{\bar{j}\bar{k}})
( \Sc^{\ \bar{j}}_{\bar{i}\ \bar{\gamma} r}   +  \tilde{\Sc}^{\ \bar{j}}_{\bar{i}\ \bar{\gamma} r} )dz^r \wedge d\bar{z}^k\wedge d\bar{z}^\gamma  \Big) ,
\end{eqnarray}
and hence the second claim, \eqref{secclaim}, is proved.  
This finishes the proof. 
\end{proof}

Using Lemma \ref{09029} and Lemma \ref{09030}, we are now able to derive  local formulae for the $L^2$-norm of curvature. 
\begin{thm} \label{KaehlerIntEst}
Let $(M,g(t))_{t\in[0,T)}$ be a smooth solution to the K\"ahler-Ricci flow $(\ref{Kahler})$ on a closed K\"ahler manifold $(M,\omega_0)$ with complex dimension $n$ and $N\subseteq M$ be a compact,connected,  complex $n$-dimensional manifold with smooth boundary $\partial N$ (possibly empty). Then for any $0\leqslant t<T$, we have

\begin{eqnarray}\label{09031001}
  \int_N |\Rm_{g(t)}|^2_{g(t)}dV_{g(t)}&&=\int_N \Sc^2_{g(t)}dV_{g(t)}+ \int_N(|\Rm_{g_0}|^2_{g_0}-\Sc^2_{g_0})dV_{g_0}+C(t)\cr
&&{\ \ \  -\frac{1}{(n-2)!}\int_N\sum^{n-3}_{j=0}C^{j}_{n-2} \rho^2_{g_0}\wedge\omega_0^j\wedge(-t\rho_0)^{n-2-j}}\cr
&&{\ \ \ -\frac{1}{(n-2)!}\int_N\sum^{n-3}_{j=0}C^{j}_{n-2} (\rho_{g_0}^2-2c_2(\omega_0))\wedge\omega_0^j\wedge(-t\rho_0))^{n-2-j}} , 
\end{eqnarray}
where 
\begin{equation}\label{090310201}
\begin{split}
C(t)&=-\int_{\partial N}\beta(t)\wedge\frac{\omega^{n-2}(t)}{(n-2)!}-\int_{\partial N}\frac{1}{2}d^cf(t)\wedge\sqrt{-1}\partial\bar{\partial}f(t)\wedge\frac{\omega^{n-2}(t)}{(n-2)!}-\int_{\partial N}d^c f(t)\wedge \rho_{g_0}\wedge\frac{\omega^{n-2}(t)}{(n-2)!}\\
&\ \ \ -\frac{1}{(n-2)!}\int_{\partial N}\sum^{n-3}_{j=0}\sum^{n-2-j-1}_{i=0}C^{j}_{n-2}C^{i}_{n-2-j} (\frac{1}{2}d^c\varphi(t))\wedge(\rho_{g_0}^2-2c_2(\omega_0))\wedge\omega_0^j\wedge(\sqrt{-1}\partial\bar{\partial}\varphi(t))^{n-3-j-i}\\
&\ \ \ -\frac{1}{(n-2)!}\int_{\partial N}\sum^{n-3}_{j=0}\sum^{n-2-j-1}_{i=0}C^{j}_{n-2}C^{i}_{n-2-j} (\frac{1}{2}d^c\varphi(t))\wedge\rho_{g_0}^2\wedge\omega_0^j\wedge(\sqrt{-1}\partial\bar{\partial}\varphi(t))^{n-3-j-i},
\end{split}
\end{equation}
and $\beta(t)$ comes from $(\ref{09100})$. 
\end{thm}
\begin{proof}
First, from Lemma \ref{09029}, we have
\begin{equation}\label{090080}
\begin{split}
&\ \ \ (|\Rm_{g(t)}|_{g(t)}^2-\Sc_{g(t)}^2)\ dV_{g(t)}\\
&= (|\Rm_{g_0}|_{g_0}^2-\Sc^2_{g_0})\ dV_{g_0}-\sqrt{-1}\partial\bar{\partial}f(t)\wedge\sqrt{-1}\partial\bar{\partial}f(t)\wedge\frac{\omega^{n-2}(t)}{(n-2)!}\cr
&\ \ \ -2\sqrt{-1}\partial\bar{\partial}f(t)\wedge \rho_{g_0}\wedge\frac{\omega^{n-2}(t)}{(n-2)!}-\frac{1}{(n-2)!}\sum^{n-3}_{j=0}C^{j}_{n-2} \rho^2_{g_0}\wedge\omega_0^j\wedge\Omega_t^{n-2-j}\cr
&\ \ \ -\sqrt{-1}\partial\bar{\partial}\alpha(t)\wedge\frac{\omega^{n-2}(t)}{(n-2)!}-\frac{1}{(n-2)!}\sum^{n-3}_{j=0}C^{j}_{n-2} (\rho_{g_0}^2-2c_2(\omega_0))\wedge\omega_0^j\wedge\Omega_t^{n-2-j}\cr
&=(|\Rm_{g_0}|_{g_0}^2-\Sc^2_{g_0})\ dV_{g_0}-\sqrt{-1}\partial\bar{\partial}f(t)\wedge\sqrt{-1}\partial\bar{\partial}f(t)\wedge\frac{\omega^{n-2}(t)}{(n-2)!}\cr
&\ \ \ -2\sqrt{-1}\partial\bar{\partial}f(t)\wedge \rho_{g_0}\wedge\frac{\omega^{n-2}(t)}{(n-2)!}-\frac{1}{(n-2)!}\sum^{n-3}_{j=0}C^{j}_{n-2} \rho^2_{g_0}\wedge\omega_0^j\wedge(-t\rho_0+\sqrt{-1}\partial\bar{\partial} \varphi(t))^{n-2-j}\cr
&\ \ \ -\sqrt{-1}\partial\bar{\partial}\alpha(t)\wedge\frac{\omega^{n-2}(t)}{(n-2)!}-\frac{1}{(n-2)!}\sum^{n-3}_{j=0}C^{j}_{n-2} (\rho_{g_0}^2-2c_2(\omega_0))\wedge\omega_0^j\wedge(-t\rho_0+\sqrt{-1}\partial\bar{\partial} \varphi(t))^{n-2-j}\cr
&=(|\Rm_{g_0}|_{g_0}^2-\Sc^2_{g_0})\ dV_{g_0}-\sqrt{-1}\partial\bar{\partial}f(t)\wedge\sqrt{-1}\partial\bar{\partial}f(t)\wedge\frac{\omega^{n-2}(t)}{(n-2)!}\cr
&\ \ \ -2\sqrt{-1}\partial\bar{\partial}f(t)\wedge \rho_{g_0}\wedge\frac{\omega^{n-2}(t)}{(n-2)!}-\frac{1}{(n-2)!}\sum^{n-3}_{j=0}C^{j}_{n-2} \rho^2_{g_0}\wedge\omega_0^j\wedge(-t \rho_0)^{n-2-j}\cr
&\ \ \ -\sqrt{-1}\partial\bar{\partial}\alpha(t)\wedge\frac{\omega^{n-2}(t)}{(n-2)!}-\frac{1}{(n-2)!}\sum^{n-3}_{j=0}C^{j}_{n-2} (\rho_{g_0}^2-2c_2(\omega_0))\wedge\omega_0^j\wedge(-t\rho_0)^{n-2-j}\cr
&\ \ \ -\frac{1}{(n-2)!}\sum^{n-3}_{j=0}\sum^{n-2-j-1}_{i=0}C^{j}_{n-2}C^{i}_{n-2-j} \rho^2_{g_0}\wedge\omega_0^j\wedge(-t\rho_0)^{i}\wedge(\sqrt{-1}\partial\bar{\partial} \varphi(t))^{n-2-j-i}\\
&\ \ \ -\frac{1}{(n-2)!}\sum^{n-3}_{j=0}\sum^{n-2-j-1}_{i=0}C^{j}_{n-2}C^{i}_{n-2-j} (\rho_{g_0}^2-2c_2(\omega_0))\wedge\omega_0^j\wedge(-t \rho_0)^i\wedge\sqrt{-1}\partial\bar{\partial} \varphi(t))^{n-2-j-i}
\end{split}
\end{equation}
where $\alpha(t)$ comes from $(\ref{09004})$ (see Lemma \ref{09030} for details), and $f(t)$ and ${\Omega_t=-t{ \rho_0}+\sqrt{-1}\partial\bar{\partial} \varphi(t)}$ come from \eqref{rickeqns}

Integrating the above equality on $N$ and using integration by parts, we deduce
\begin{equation}\label{090310}
\begin{split}
\int_N |\Rm_{g(t)}|^2_{g(t)}dV_{g(t)}&=\int_N \Sc^2_{g(t)}dV_{g(t)}+ \int_N(|\Rm_{g_0}|^2_{g_0}-\Sc^2_{g_0})dV_{g_0}+C(t)\\
&\ \ \ -\frac{1}{(n-2)!}\int_N\sum^{n-3}_{j=0}C^{j}_{n-2} \rho^2_{g_0}\wedge\omega_0^j\wedge(-t{ \rho_0})^{n-2-j}\\
&\ \ \ -\frac{1}{(n-2)!}\int_N\sum^{n-3}_{j=0}C^{j}_{n-2} (\rho_{g_0}^2-2c_2(\omega_0))\wedge\omega_0^j\wedge(-t{ \rho_0}))^{n-2-j},
\end{split}
\end{equation}
where 
\begin{equation}\label{0903102}
\begin{split}
C(t)&=-\int_{\partial N}\beta(t)\wedge\frac{\omega^{n-2}(t)}{(n-2)!}-\int_{\partial N}\frac{1}{2}d^cf(t)\wedge\sqrt{-1}\partial\bar{\partial}f(t)\wedge\frac{\omega^{n-2}(t)}{(n-2)!}-\int_{\partial N}d^c f(t)\wedge \rho_{g_0}\wedge\frac{\omega^{n-2}(t)}{(n-2)!}\cr
&\ \ \ -\frac{1}{(n-2)!}\int_{\partial N}\sum^{n-3}_{j=0}\sum^{n-2-j-1}_{i=0}C^{j}_{n-2}C^{i}_{n-2-j} (\frac{1}{2}d^c\varphi(t))\wedge(\rho_{g_0}^2-2c_2(\omega_0))\wedge\omega_0^j\wedge(\sqrt{-1}\partial\bar{\partial}\varphi(t))^{n-3-j-i}\cr
&\ \ \ -\frac{1}{(n-2)!}\int_{\partial N}\sum^{n-3}_{j=0}\sum^{n-2-j-1}_{i=0}C^{j}_{n-2}C^{i}_{n-2-j} (\frac{1}{2}d^c\varphi(t))\wedge\rho_{g_0}^2\wedge\omega_0^j\wedge(\sqrt{-1}\partial\bar{\partial}\varphi(t))^{n-3-j-i},
\end{split}
\end{equation}
and $\beta(t)$ comes from $(\ref{09100})$ 
as claimed.
\end{proof}
  In the following, we see that the constant $C(t)$ may be uniformly estimated from above  if we are in the setting ({\bf B}). 
\begin{thm}\label{KaehlerIntEstCor}
Let $(M,g(t))_{t\in[0,T)}$ with $T<\infty$ be a smooth solution to the K\"ahler-Ricci flow $(\ref{Kahler})$ on a closed K\"ahler manifold $(M,\omega_0)$ with complex dimension $n$ and  $N  \subseteq M$ be   a  smooth open connected  real $2n$-dimensional   submanifold, where    $N  \subseteq M,$ $\Omega$   and $(M^{n},g(t))_{t\in [0,T)}$    are  as in ({\bf B}). 
 Then $C(t)$ from the Theorem above, Theorem \ref{KaehlerIntEst}, satisfies
   $\sup\limits_{t\in [0,T)} |C(t)|  
< \infty.$
\end{thm}

\begin{proof}
Take a covering of $N$ by a collection of balls $Z:= \cup_{i=1}^N B_{r_i}(p_i)$ at time zero, such that 
$ \overline{B_{5r_i}(p_i) }\subseteq \Omega$  for all $i=1, \ldots,N$ and so that each ball $B_{5r_i}(p_i)$ admits holomorphic  geodesic coordinates.
We define $\ti Z:= \cup_{i=1}^N B_{4r_i}(p_i),$ 
  $\hat Z:=  \cup_{i=1}^N   \overline{B_{5r_i}(p_i)}.$
The curvatures and all covariant derivatives thereof
are uniformly bounded in time on $\hat Z:= 
\cup_{i=1}^N   \overline{B_{5r_i}(p_i)}$  by assumption (the constants depending on the order of the covariant derivative).
This means that the metric is bounded uniformly in time from above and below by the time zero metric on  $\ti Z$.
Furthermore,  the euclidean norm of any spatial derivative of any order of the metric  in one of the coordinate charts on $B_{4r_i}(p_i)$   is   also bounded 
uniformly     from above  on $[0,T)$ as one sees in the proof of Theorem 8.1 in \cite{HamSing}. 
Hence  the norm with respect to $g(t)$ or $g_0$ of all the terms $\beta(t)$, $d^c f(t),$ $\partial \overline{\partial}f(t)$ $d^c \varphi$, $\partial \bar{\partial} \varphi(t),$ and so on, appearing in $C(t)$, are all bounded uniformly by constants independent of $t\in [0,T)$  on  this   neighborhood and hence $\sup\limits_{t\in [0,T)} |C(t)|< \infty$ as claimed. 

\end{proof}

\section{Appendix A}
{\it Proof of Lemma \ref{400} and Lemma \ref{09001}.} Let $(z_1,\cdots,z_n)$ be a local complex coordinates such that
\begin{equation*}\label{092201}
\begin{split}
&\ \omega=\sum_{i=1}^{n} \frac{\sqrt{-1}}{2\pi}dz^i\wedge d\bar{z}^i,\ \ \ \Omega^j_i=\frac{\sqrt{-1}}{2\pi}g^{j\bar{p}}\Sc_{i\bar{p}k\bar{l}}\ dz^k\wedge d\bar{z}^l=\frac{\sqrt{-1}}{2\pi}\Sc_{i\bar{j}k\bar{l}}\ dz^k\wedge d\bar{z}^l,\\
&\ c_2(\omega)=\frac{1}{2}\sum^{n}_{i,j=1}(\Omega^i_i \wedge\Omega^j_j-\Omega^i_j\wedge\Omega^j_i).
\end{split}
\end{equation*}
Then we have
\begin{equation}\label{092202}
\begin{split}
\rho_g\wedge\rho_g\wedge\omega^{n-2}&=\Omega^i_i\wedge\Omega^j_j\wedge\omega^{n-2}\\
&=(\frac{\sqrt{-1}}{2\pi})^2\Sc_{i\bar{i}k\bar{p}}\Sc_{j\bar{j}m\bar{l}}\ dz^k\wedge d\bar{z}^p\wedge dz^m\wedge d\bar{z}^l\wedge\omega^{n-2}\\
&=(\frac{\sqrt{-1}}{2\pi})^2(\Sc_{i\bar{i}k\bar{k}}\Sc_{j\bar{j}l\bar{l}}-\Sc_{i\bar{i}k\bar{l}}\Sc_{j\bar{j}l\bar{k}})dz^k\wedge d\bar{z}^k\wedge dz^l\wedge d\bar{z}^l\wedge\omega^{n-2}\\
&=\frac{1}{n(n-1)}(\Sc_{i\bar{i}k\bar{k}}\Sc_{j\bar{j}l\bar{l}}-\Sc_{i\bar{i}k\bar{l}}\Sc_{j\bar{j}l\bar{k}})\ \omega^{n}\\
&=\frac{1}{n(n-1)}(\Sc_{g}^2-|\Ric_{g}|_{g}^2)\ \omega^{n}
\end{split}
\end{equation}
and
\begin{equation}\label{0922020}
\begin{split}
(\rho_g\wedge \rho_g-2c_2(\omega))\wedge\omega^{n-2}&=\Omega^i_j\wedge\Omega^j_i\wedge\omega^{n-2}\\
&=(\frac{\sqrt{-1}}{2\pi})^2(\Sc_{i\bar{j}k\bar{k}}\Sc_{j\bar{i}l\bar{l}}-\Sc_{i\bar{j}k\bar{l}}\Sc_{j\bar{i}l\bar{k}})dz^k\wedge d\bar{z}^k\wedge dz^l\wedge d\bar{z}^l\wedge\omega^{n-2}\\
&=\frac{1}{n(n-1)}(\Sc_{i\bar{j}k\bar{k}}\Sc_{j\bar{i}l\bar{l}}-\Sc_{i\bar{j}k\bar{l}}\Sc_{j\bar{i}l\bar{k}})\ \omega^{n}\\
&=\frac{1}{n(n-1)}(|\Ric_{g}|_{g}^2-|\Rm_{g}|_{g}^2)\ \omega^{n}.
\end{split}
\end{equation}
Hence we get Lemma \ref{09001}. If $M$ is a closed K\"ahler manifolds with complex dimension $n$,
 then Lemma \ref{400}  now follows from Stoke's formula and the fact that
  $c_1(M)=[\rho_g]$ and $c_2(M)=[c_2(\omega)].$ \QEDB

{\it Conflict of interest statement}: There is no conflict of interest. \\
{\it Data availability statement }: No datasets were generated or analysed during the current study. \\

\end{document}